\documentclass[reqno,10pt, centertags]{amsart}
\usepackage{amssymb,upref,esint}
\usepackage{hyperref}

\newcommand*{\mailto}[1]{\href{mailto:#1}{\nolinkurl{#1}}}
\newcommand{\arxiv}[1]{\href{http://arxiv.org/abs/#1}{arXiv:#1}}

\makeatletter
\def\theequation{\@arabic\c@equation}

\newcommand{\bbR}{{\mathbb{R}}}

\newcommand{\bbC}{{\mathbb{C}}}

\newcommand{\cB}{{\mathcal B}}
\newcommand{\cC}{{\mathcal C}}
\newcommand{\cD}{{\mathcal D}}

\newcommand{\cF}{{\mathcal F}}
\newcommand{\cG}{{\mathcal G}}
\newcommand{\cH}{{\mathcal H}}

\newcommand{\cL}{{\mathcal L}}
\newcommand{\cM}{{\mathcal M}}
\newcommand{\cN}{{\mathcal N}}

\newcommand{\no}{\nonumber}
\newcommand{\lb}{\label}
\newcommand{\f}{\frac}

\newcommand{\ol}{\overline}

\newcommand{\wti}{\widetilde}

\newcommand{\tr}{\text{\rm{tr}}}

\newcommand{\ran}{\text{\rm{ran}}}
\newcommand{\ind}{\text{\rm{ind}}}

\newcommand{\dom}{\text{\rm{dom}}}

\newcommand{\bi}{\bibitem}

\renewcommand{\Im}{\text{\rm Im}}

\numberwithin{equation}{section}

\newtheorem{theorem}{Theorem}[section]
\newtheorem{lemma}[theorem]{Lemma}
\newtheorem{corollary}[theorem]{Corollary}
\newtheorem{proposition}[theorem]{Proposition}
\theoremstyle{definition}
\newtheorem{definition}[theorem]{Definition}
\newtheorem{hypothesis}[theorem]{Hypothesis}
\newtheorem{remark}[theorem]{Remark}
\newtheorem{example}[theorem]{Example}

\begin{document}

\title[Dirichlet-to-Neumann Maps, Weyl Functions, and the Index]{Dirichlet-to-Neumann Maps, Abstract Weyl--Titchmarsh $M$-Functions, and a Generalized Index of Unbounded Meromorphic 
Operator-Valued Functions}

\author[J.\ Behrndt]{Jussi Behrndt}  
\address{Institut f\"ur Numerische Mathematik, Technische Universit\"at 
Graz, Steyrergasse 30, 8010 Graz, Austria}  
\email{\mailto{behrndt@tugraz.at}}
\urladdr{\url{http://www.math.tugraz.at/~behrndt/}}

\author[F.\ Gesztesy]{Fritz Gesztesy} 
\address{Department of Mathematics,
University of Missouri, Columbia, MO 65211, USA}
\email{\mailto{gesztesyf@missouri.edu}}
\urladdr{\url{https://www.math.missouri.edu/people/gesztesy}}

\author[H.\ Holden]{Helge Holden}
\address{Department of Mathematical Sciences,
Norwegian University of
Science and Technology, NO--7491 Trondheim, Norway}
\email{\mailto{holden@math.ntnu.no}}
\urladdr{\href{http://www.math.ntnu.no/~holden/}{http://www.math.ntnu.no/\~{}holden/}}

\author[R.\ Nichols]{Roger Nichols}
\address{Mathematics Department, The University of Tennessee at Chattanooga, 
415 EMCS Building, Dept. 6956, 615 McCallie Ave, Chattanooga, TN 37403, USA}
\email{\mailto{Roger-Nichols@utc.edu}}
\urladdr{\url{http://www.utc.edu/faculty/roger-nichols/index.php}}

\date{\today}
\thanks{ J.B.\ and F.G.\ gratefully acknowledge support by the Austrian Science Fund (FWF), project
P 25162-N26. F.G.\ and H.H.\ were supported in part by the Research Council of Norway. 
R.N.\ gratefully acknowledges support from a UTC College of Arts and Sciences RCA Grant.} 

\begin{abstract}
We introduce a generalized index for certain meromorphic, unbounded, 
operator-valued functions. The class of functions is chosen such that 
energy parameter dependent Dirichlet-to-Neumann 
maps associated to uniformly elliptic partial differential operators, particularly, non-self-adjoint 
Schr\"odinger operators, on bounded Lipschitz domains, and abstract operator-valued 
Weyl--Titchmarsh $M$-functions and
Donoghue-type $M$-functions corresponding to closed extensions of symmetric operators belong to it.

The principal purpose of this paper is to prove index formulas that relate the difference of the 
algebraic multiplicities of the discrete eigenvalues of Robin realizations of non-self-adjoint Schr\"{o}dinger operators,
and more abstract pairs of closed operators in Hilbert spaces with the generalized index of the corresponding energy dependent
Dirichlet-to-Neumann maps and abstract Weyl--Titchmarsh $M$-functions, respectively.
\end{abstract}

\maketitle


{\scriptsize{\tableofcontents}}

\section{Introduction}

The principal purpose of this paper is to prove index formulas that relate the
algebraic multiplicities of the discrete eigenvalues of closed operators in Hilbert spaces with a certain
generalized index of a class of meromorphic, unbounded, closed, operator-valued functions, which have constant domains and are not necessarily Fredholm. In the following, we shall briefly illustrate the index formulas in our main applications and familiarize the reader with the structure of this article.

Let us first consider the Schr\"{o}dinger differential expression 
\begin{equation}\label{lintro}
\mathcal L=-\Delta+q 
\end{equation}
on a bounded Lipschitz domain $\Omega\subset\mathbb R^n$, $n\geq 2$, with a complex-valued, bounded, measurable potential $q\in L^\infty(\Omega)$.
Denote by $A_D$ the Dirichlet realization of $\mathcal L$ in $L^2(\Omega)$ and let $A_\Theta$ be a closed realization of $\mathcal L$ subject to
Robin-type boundary conditions of the form
\begin{equation}\label{robinintro}
 \Theta\gamma_D f=\gamma_N f,
\end{equation}
where $\gamma_D$ and $\gamma_N$ denote the Dirichlet and Neumann trace operator, and $\Theta$ is a bounded operator in $L^2(\partial\Omega)$;
for precise definitions of the trace maps and the operators $A_D$ and $A_\Theta$ we refer to Section~\ref{sec3}.
We emphasize that the differential expression \eqref{lintro} is non-symmetric and hence the Dirichlet and Robin realization $A_D$ and $A_\Theta$ are non-self-adjoint, and 
that, in addition, also the parameter $\Theta$ in the Robin boundary condition in \eqref{robinintro} is non-self-adjoint in general. 
Since the Lipschitz domain $\Omega$ is bounded, the spectra of the operators $A_D$ and $A_\Theta$ consist of isolated eigenvalues with finite algebraic multiplicities.
As one of our main results we show that the algebraic multiplicities $m_a(z_0;A_D)$ and $m_a(z_0;A_\Theta)$ of an eigenvalue $z_0$ of $A_D$ and $A_\Theta$ satisfy the generalized 
index formula
\begin{equation}\label{indexintro}
\wti{\ind}_{C(z_0; \varepsilon)}(D(\cdot)-\Theta)=m_a(z_0;A_D)-m_a(z_0;A_\Theta),
\end{equation}
where the generalized index $\wti{\ind}_{C(z_0; \varepsilon)}(\,\cdot\,)$ is defined below in \eqref{indexdef}, and $D(\cdot)$ denotes the energy parameter-dependent Dirichlet-to-Neumann map associated to the differential expression $\mathcal L$. The index formula \eqref{indexintro} remains valid for points $z_0$ in the resolvent set of $\rho(A_D)$ or $\rho(A_\Theta)$, in which case $m_a(z_0;A_D)=0$ or $m_a(z_0;A_\Theta)=0$, respectively.
However, since the values $D(z)$, $z \in\rho(A_D)$, of the Dirichlet-to-Neumann map are unbounded 
operators in $L^2(\partial\Omega)$,  
the classical concept of an index for a meromorphic, bounded, Fredholm operator-valued function  as introduced in \cite{GS71} (see also \cite[Chapter~XI.9]{GGK90} 
and \cite[Chapter~4]{GL09}) does not apply to $D(\cdot)-\Theta$
on the left-hand side of \eqref{indexintro}. 

Instead, it is necessary to specify a suitable class of meromophic operator-valued functions $M(\cdot)$ with values in the set of unbounded closed operators such that
on one hand the function $D(\cdot)-\Theta$ in \eqref{indexintro} is contained in this class, and on the other hand the generalized index
 \begin{equation}\label{indexdef}
\wti{\ind}_{C(z_0; \varepsilon)}(M(\cdot)) := \tr\bigg(\f{1}{2\pi i} 
\ointctrclockwise_{C(z_0; \varepsilon)} d\zeta \, 
\ol{M'(\zeta)} M(\zeta)^{-1}\bigg)      
\end{equation}  
is well-defined; here $C(z_0; \varepsilon)$ is the counterclockwise 
oriented circle centered at $z_0$ with radius $\varepsilon>0$ sufficiently small, and $\ol{M'(\zeta)}$ denotes the closure of the derivative of $M(\cdot)$ at $\zeta$.
This is the main purpose of the preliminary Section~\ref{sec2}, which is inspired by considerations in \cite{BGHN16} and \cite{GHN15}.
Here we collect a set of assumptions and define a class of meromorphic, unbounded, closed, operator-valued functions,
which are not necessarily Fredholm, such that 
the functions $\ol{M'(\cdot)}$ and $M(\cdot)^{-1}$ in the integrand in \eqref{indexdef} are both finitely meromorphic (see \cite{GGK90}, \cite{GL09}), 
and hence definition \eqref{indexdef} turns out to be meaningful. Although the generalized index in \eqref{indexdef} may not be integer-valued in general
(in contrast to the classical index, where the operator-valued version of the argument principle from \cite{GS71} or \cite[Theorem\ 4.4.1]{GL09} applies)
in our main applications \eqref{indexintro} and \eqref{indiwowintro} below it certainly is, since the 
right-hand side equals an integer.

The main objective of Section~\ref{sec3} is to prove the index formula \eqref{indexintro} in 
Theorem~\ref{k1}. Besides the differential expression $\mathcal L=-\Delta+q$ we also
consider the formal adjoint expression $\widetilde{\mathcal L}=-\Delta+\ol q$ and obtain an 
analogous index formula for the algebraic multiplicities of the eigenvalues of 
$A_D^*$ and $A_\Theta^*$ in Theorem~\ref{k2}.
The main ingredient in the proof of the index formula \eqref{indexintro} is the Krein-type resolvent formula in Theorem~\ref{k1} in which the difference of the 
resolvents of $A_\Theta$ and $A_D$ in $L^2(\Omega)$
is traced back to the boundary space $L^2(\partial\Omega)$ and the perturbation term 
$D(\cdot)-\Theta$. Such resolvent formulas are well-known
for the symmetric case (see, e.g.,  
\cite{AB09}, \cite{BL07}, \cite{BM14}, \cite{BMNW08}, 
\cite{GM11}, \cite{Ma10}, \cite{PR09}, \cite{Po12}) and in the context of dual pairs related formulas 
can be found, for instance, in \cite{BGW09} and \cite{MM02}; the Dirichlet-to-Neumann map 
$D(\cdot)$ has atracted a lot of attention in the recent past (see, e.g., \cite{AB09}--\cite{AM12}, \cite{BGMM16}--\cite{BR15a}, \cite{GM08}, \cite{GM11}, \cite{Po08}, \cite{PR09}, and the references therein). Although formally the index formula \eqref{indexintro} is an immediate 
consequence of the Krein-type resolvent formula we wish to emphasize that it is necessary to 
verify that the generalized index \eqref{indexdef} is well-defined for the function $D(\cdot)-\Theta$.
In fact, a somewhat subtle analysis is required in this context, and the key difficulty is to show 
that $(D(\cdot)-\Theta)^{-1}$ is a finitely meromorphic function (cf. Lemma~\ref{dt1}).

Besides the index formula for Robin realizations of $\mathcal L$ in Section~\ref{sec3}, we also discuss a slightly more abstract situation in Section~\ref{sec4}. Here
it is assumed that $B_1$ and $B_2$ are closed operators in a Hilbert space $\mathfrak H$ which are both extensions of a common underlying densely defined, symmetric
operator $S$. We shall use the abstract concept of boundary triples (see, e.g.,  \cite{BL12}, \cite{Br76}, \cite{BGP08}, \cite{DM91}, \cite{DM95}, \cite{GG91}, \cite{Ko75}) to parametrize $B_1$ and $B_2$ in the form
\begin{equation}\label{bbintro}
 B_1=S^*\upharpoonright\ker(\Gamma_1-\Theta_1\Gamma_0), \quad
 B_2=S^*\upharpoonright\ker(\Gamma_1-\Theta_2\Gamma_0),
\end{equation}
where $\Gamma_0$ and $\Gamma_1$ are linear maps from $\dom(S^*)$ into a boundary space $\mathcal G$ and $\Theta_1$ and $\Theta_2$ are closed operators in $\mathcal G$. 
Let $M(\cdot)$ denote the Weyl--Titchmarsh function corresponding to the boundary triple $\{\mathcal G,\Gamma_0,\Gamma_1\}$. Our goal in Section~\ref{sec4} is 
to prove the index formula
\begin{equation}\label{indiwowintro}
\wti{\ind}_{C(z_0; \varepsilon)} \big(\Theta_1 - M(\cdot)\big) -
\wti{\ind}_{C(z_0; \varepsilon)} \big(\Theta_2 - M(\cdot)\big)
= m_a\big(z_0;B_1\big) -m_a\big(z_0; B_2\big),
\end{equation}
in which the generalized index of the functions $\Theta_1 - M(\cdot)$ and $\Theta_2 - M(\cdot)$ is related to the algebraic multiplicities of a discrete eigenvalue $z_0$ of
$B_1$ and $B_2$ (the formula is also valid for points $z_0$ in the resolvent set of $B_1$ or $B_2$, in which case $m_a(z_0;B_1)=0$ or $m_a(z_0;B_2)=0$, respectively). 
In contrast to the index formula \eqref{indexintro} in Section~\ref{sec3}, here the values of the Weyl--Titchmarsh function $M(\cdot)$ are bounded operators, but the operator-valued 
parameters $\Theta_1$ and $\Theta_2$ are in general unbounded, closed operators. However, the strategy and the difficulties in the proof of the index formula in Theorem~\ref{indithm}
are similar to those in Section~\ref{sec3}: One first has to verify that the generalized index is 
well-defined for the functions $\Theta_1 - M(\cdot)$ and $\Theta_2 - M(\cdot)$ (again the key 
difficulty is to show that the inverses $(\Theta_1-M(\cdot))^{-1}$ and  $(\Theta_1-M(\cdot))^{-1}$ are finitely meromophic at a discrete eigenvalue of $B_1$ and $B_2$, respectively) and then a 
Krein-type resolvent formula (see, e.g., \cite{AB09}, \cite{AP04}, \cite{BL07},  \cite{BM14}, 
\cite{BR15a}, \cite{BGW09}--\cite{BNMW14},  \cite{DHMS09}--\cite{DMT88},  
\cite{GKMT01}--\cite{GM11}, \cite{GT00}, \cite{HMM13}, \cite{KO77}, \cite{KO78},  \cite{LT77}, \cite{Ma92a}, \cite{Ma92b}, \cite{Sa65}, and the references cited therein) yields the index formula \eqref{indiwowintro}. 

To ensure a self-contained presentation in Section~\ref{sec4}, we have added a short 
Appendix \ref{sA} on the abstract concept of boundary triples and their Weyl--Titchmarsh functions.
In this appendix we also establish the connection to abstract Donoghue-type $M$-functions studied in \cite{GKMT01}, \cite{GMT98}, \cite{GNWZ15}, \cite{GT00}, so that the index formula \eqref{indiwowintro} can also be interpreted in the framework of Donoghue-type $M$-functions.

Finally, we summarize the basic notation used in this paper: $\mathcal H$, $\mathfrak H$, and 
$\mathcal G$ denote separable complex 
Hilbert spaces with scalar products $(\,\cdot\,,\,\cdot\,)_{\cH}$, $(\,\cdot\,,\,\cdot\,)_{\mathfrak H}$, 
and $(\,\cdot\,,\,\cdot\,)_{\mathcal G}$, linear in the first entry, 
respectively. 
The Banach spaces of bounded, compact, and trace class (linear) operators in $\cH$ are denoted by 
$\cB(\cH)$, $\cB_\infty(\cH)$, and $\cB_1(\cH)$, respectively. The subspace of all finite rank operators will be abbreviated by $\cF(\cH)$. The analogous notation $\cB(\cH,\cG)$
will be used for bounded operators between the Hilbert spaces $\cH$ and $\cG$.
The set of densely defined, closed, linear operators in $\cH$ will be denoted by $\cC(\cH)$.
For a linear operator $T$ we denote by $\dom(T)$, $\ran(T)$ and $\ker(T)$ the domain, range, and kernel, respectively.
If $T$ is closable, the closure is denoted by $\overline T$. The spectrum, point spectrum, continuous spectrum, residual spectrum, and resolvent set of a closed  
operator $T\in\mathcal C(\mathcal H)$  will be denoted by $\sigma(T)$, $\sigma_p(T)$, $\sigma_c(T)$, $\sigma_r(T)$, and $\rho(T)$; the 
discrete spectrum of $T$ consists of eigenvalues of $T$ with finite algebraic multiplicity which are isolated in 
$\sigma(T)$, this set is  
abbreviated by $\sigma_d(T)$. 
For the algebraic multiplicity of an eigenvalue $z_0\in\sigma_d(T)$ we write $m_a(z_0;T)$ and
we set $m_a(z_0;T)=0$ if $z_0\in\rho(T)$.
Furthermore,
$\tr_{\cH}(T)$ denotes the trace of a trace class operator $T \in \cB_1(\cH)$. 
The symbol $\dotplus$ denotes a direct (but not necessary orthogonal direct) sum decomposition 
in connection with subspaces of Banach spaces.

\section{On the Notion of a Generalized Index of Meromorphic Operator-Valued Functions} 
\label{sec2}

Let $\cH$ be a separable complex Hilbert space,
assume that $\Omega \subseteq \mathbb C$ is an open set, and let $M(\cdot)$ be a $\cB(\cH)$-valued 
meromorphic function on $\Omega$ that has the norm convergent Laurent expansion around 
$z_0 \in \Omega$ of the form
\begin{equation}\label{mmm}
 M(z) = \sum_{k= - N_0}^\infty (z - z_0)^{k} M_k(z_0),\quad z\in D(z_0;\varepsilon_0)\backslash\{z_0\}, 
\end{equation}
where $M_k(z_0) \in \cB(\cH)$, $k\in\mathbb Z$, $k\geq -N_0$ and $\varepsilon_0>0$ is sufficiently small such that
the punctured open disc 
\begin{equation} 
D(z_0;\varepsilon_0)\backslash\{z_0\}=\{z\in\mathbb C \,|\, 0< \vert z-z_0\vert<\varepsilon_0\} 
\end{equation}
is contained in $\Omega$. The principal part ${\rm pp}_{z_0} \{M(z)\}$ of $M(\cdot)$ at $z_0$ is defined as the finite sum 
\begin{equation}\label{koko}
 {\rm pp}_{z_0} \{M(z)\}=\sum_{k= - N_0}^{-1} (z - z_0)^{k} M_k(z_0).
\end{equation}

\begin{definition} \label{d2.1}
Let $\Omega\subseteq\bbC$ be an open set and let $M(\cdot)$ be a $\cB(\cH)$-valued 
meromorphic function on $\Omega$. Then  
$M(\cdot)$ is called {\it finitely meromorphic at $z_0 \in\Omega$} if $M(\cdot)$ is 
analytic on the punctured disk $D(z_0;\varepsilon_0)\backslash\{z_0\} \subset \Omega$ 
with sufficiently small $\varepsilon_0 > 0$, and the principal part ${\rm pp}_{z_0} \{M(z)\}$ of $M(\cdot)$ at $z_0$ is of finite rank, that is, 
the principal part of $M(\cdot)$ is of the type \eqref{koko}, and one has 
\begin{equation}\label{lkj}
M_k(z_0) \in \cF(\cH), \quad -N_0 \leq k \leq -1. 
\end{equation}
The function $M(\cdot)$ is called {\it finitely meromorphic on $\Omega$} if it is meromorphic 
on $\Omega$ and finitely meromorphic at each of its poles.
\end{definition}

Assume that $M_j(\cdot)$, $j=1,2$, are $\cB(\cH)$-valued 
meromorphic functions on $\Omega$ that are both finitely meromorphic at $z_0 \in \Omega$,
choose $\varepsilon_0 > 0$ such that \eqref{mmm} and \eqref{lkj} hold for both functions $M_j(\cdot)$, and let $0<\varepsilon<\varepsilon_0$.
Then by \cite[Lemma\ XI.9.3]{GGK90} or \cite[Proposition\ 4.2.2]{GL09} also the functions 
$M_1(\cdot) M_2(\cdot)$ and $M_2(\cdot) M_1(\cdot)$ are finitely meromorphic at 
$z_0 \in \Omega$, the operators 
\begin{equation}\label{ws}
\ointctrclockwise_{C(z_0;\varepsilon)} d \zeta \, M_1(\zeta) M_2(\zeta)  \quad\text{and}\quad \ointctrclockwise_{C(z_0;\varepsilon)} d \zeta \, M_2(\zeta) M_1(\zeta)
\end{equation}
are both of finite rank and the identity
\begin{equation}\label{wss}
 {\tr}_{\cH} \bigg(\ointctrclockwise_{C(z_0;\varepsilon)} d \zeta \, 
M_1(\zeta) M_2(\zeta)\bigg) 
= {\tr}_{\cH} \bigg(\ointctrclockwise_{C(z_0;\varepsilon)} d \zeta \, 
M_2(\zeta) M_1(\zeta)\bigg)
\end{equation}
holds; here the symbol $ \ointctrclockwise$ denotes the contour integral and
$C(z_0; \varepsilon) = \partial D(z_0; \varepsilon)$ is the counterclockwise oriented circle 
with radius $\varepsilon$ centered at $z_0$.

In the next example a standard situation is discussed: the resolvent of a closed operator $T$ in the Hilbert space $\mathcal H$ is finitely meromorphic at a discrete eigenvalue (cf. \cite{GK69} or \cite{Ka80}).

\begin{example}\label{ex1}
Let $T$ be a closed operator in the Hilbert space $\mathcal H$ and let $z_0\in\sigma_d(T)$. Choose $\varepsilon_0>0$ sufficiently small such that the 
punctured disc $D(z_0;\varepsilon_0)\backslash\{z_0\}$ is contained in $\rho(T)$ and let $0<\varepsilon<\varepsilon_0$. Then the Riesz projection
\begin{equation} 
P(z_0;T)=-\frac{1}{2\pi i} \ointctrclockwise_{C(z_0;\varepsilon)} d\zeta \, (T -\zeta I_{\cH})^{-1},
\end{equation}
where as above $C(z_0; \varepsilon) = \partial D(z_0; \varepsilon)$, is a finite rank operator in $\mathcal H$
and the range of $P(z_0;T)$ coincides with the algebraic eigenspace of $T$ at $z_0$;  
in particular, one has 
\begin{equation} 
{\tr}_{\cH}(P(z_0;T)) = m_a(z_0; T).  
\end{equation} 
Furthermore, the Hilbert space $\mathcal H$ admits the direct sum decomposition 
\begin{equation}
 \mathcal H=\ran(P(z_0;T))\,\dot+\,\ran(I_{\mathcal H}-P(z_0;T))
\end{equation}
and the spaces $P(z_0;T)\mathcal H$ and $(I_{\mathcal H}-P(z_0;T))\mathcal H$ are both invariant for the closed operators $T$ and $T-z_0 I_{\mathcal H}$. Moreover, the restriction 
$T_1-z_0I_{\mathcal H}$ of $T-z_0I_{\mathcal H}$ onto the finite-dimensional
subspace $P(z_0;T)\mathcal H$ is nilpotent, that is, $(T_1-z_0 I_{\mathcal H})^{N_0}=0$ for some $N_0\in\mathbb N$ and we agree to choose the integer $N_0$ with this property minimal. The restriction 
$T_2-z_0I_{\mathcal H}$ of $T-z_0I_{\mathcal H}$ onto 
$(I_{\mathcal H}-P(z_0;T))\mathcal H$ is a boundedly invertible operator in the Hilbert space $(I_{\mathcal H}-P(z_0;T))\mathcal H$. As in  
\cite[Chapter 1, $\S$2. Proof of Theorem 2.1]{GK69} one verifies that the resolvent of $T$ in $D(z_0;\varepsilon_0)\backslash\{z_0\}$
admits a norm convergent Laurent expansion of the form
\begin{equation}\label{resex}
\begin{split}
 (T-z I_{\mathcal H})^{-1}&=-\sum_{k=-N_0}^{-1}(z-z_0)^k (T_1-z_0 I_{\mathcal H})^{-k-1} P(z_0;T)\\
                          &\quad+ \sum_{k=0}^\infty (z-z_0)^k (T_2-zI_{\mathcal H})^{-(k+1)}(I_{\mathcal H}-P(z_0;T)),
\end{split}
 \end{equation}
 and, in particular, the operators $(T_1-z_0 I_{\mathcal H})^{-k-1} P(z_0;T)$, $-N_0\leq k\leq -1$, are of finite rank. Therefore, the resolvent $z\mapsto (T-zI_{\mathcal H})$
 is finitely meromorphic at $z_0$. It also follows from the Laurent expansion \eqref{resex} that the derivatives $\tfrac{d^k}{dz^k}(T-z I_{\mathcal H})^{-1}$, $k \in\mathbb N$, are 
 finitely meromorphic at $z_0$.
\end{example}

The following example is a simple generalization and immediate consequence of Example~\ref{ex1}. The observation below will be used frequently in this paper.

\begin{example}\label{ex2}
 Let $T$ be a closed operator in the Hilbert space $\mathcal H$ and let $z_0\in\sigma_d(T)$. Assume that $\mathcal G$ is an auxiliary Hilbert space 
 and let $\gamma\in\mathcal B(\mathcal G,\mathcal H)$. Then the $\mathcal B(\mathcal G)$-valued function 
 \begin{equation}
  z\mapsto \gamma^*(T-zI_{\mathcal H})^{-1}\gamma,\quad z\in\rho(T),
 \end{equation}
is finitely meromorphic at $z_0$. Indeed, this simply follows by multiplying the Laurent expansion of the resolvent in \eqref{resex} by $\gamma^*\in\mathcal B(\mathcal H,\mathcal G)$
from the left and by $\gamma\in\mathcal B(\mathcal G,\mathcal H)$
from the right.
\end{example}

The aim of this preliminary section is to
introduce an extended notion of the index applicable to certain non-Fredholm and also unbounded 
meromorphic operator-valued functions $M(\cdot)$ in Definition~\ref{d3.6} below.
We start by collecting our assumptions on $M(\cdot)$. 

\begin{hypothesis} \lb{h3.5}
Let $\Omega \subseteq \bbC$ be open and connected, and $\cD_0 \subset \Omega$ a discrete 
set (i.e., a set without limit points in $\Omega$). Suppose that the map 
\begin{equation}
M:\Omega \backslash \cD_0 \to \cC(\cH),\quad z\mapsto M(z),
\end{equation}
takes on values in the set of densely defined, closed operators, $\cC(\cH)$, with the 
following additional properties: \\[1mm] 
 $(i)$ $\cM_0 := \dom (M(z))$ is independent of $z\in\Omega\backslash\cD_0$. \\[1mm] 
 $(ii)$ $M(z)$ is boundedly invertible, $M(z)^{-1} \in \cB(\cH)$ for all $z \in \Omega \backslash \cD_0$. \\[1mm]  
 $(iii)$ The function 
 \begin{equation}
 M(\cdot)^{-1}:\Omega\backslash\cD_0\rightarrow \cB(\cH),\quad z\mapsto M(z)^{-1}, 
 \end{equation}
 is analytic on $\Omega \backslash \cD_0$ and finitely meromorphic on $\Omega$. \\[1mm] 
 $(iv)$ For $\varphi\in \cM_0$ the function 
 \begin{equation} 
 M(\cdot)\varphi:\Omega\backslash\cD_0\rightarrow\cH,\quad z\mapsto M(z)\varphi, 
 \end{equation} 
 is analytic; in particular, the derivative $M'(z)\varphi$ exists for all $\varphi\in \cM_0$ and 
 $z\in  \Omega\backslash\cD_0$. \\[1mm] 
$(v)$ For $z\in\Omega\backslash\cD_0$, the operators $M'(z)$ defined on $\dom(M'(z))=\cM_0$, 
admit bounded continuations to operators $\overline{M'(z)}\in\cB(\cH)$, and the  
operator-valued function
\begin{equation} 
\overline{M'(\cdot)}:\Omega \backslash \cD_0 \to \cB(\cH),\quad z\mapsto \overline{M'(z)}, 
\end{equation} 
is analytic on $\Omega \backslash \cD_0$ and finitely meromorphic on $\Omega$. 
\end{hypothesis}

Granted Hypothesis \ref{h3.5} it follows that the maps 
\begin{equation}
z\mapsto \ol{M'(z)} M(z)^{-1}, \quad z\mapsto M(z)^{-1}\ol{M'(z)}
\end{equation}
are finitely meromorphic and hence identity \eqref{wss} applies. This leads to the following definition of a generalized index of $M(\cdot)$, which extends the
notion of an index for finitely meromorphic $\mathcal B(\mathcal H)$-valued functions employed in \cite{GS71} and, for instance, in \cite{GGK90, GL09} (cf. \cite[Definition 4.2]{BGHN16}).

\begin{definition} \lb{d3.6} 
Assume Hypothesis \ref{h3.5}, let $z_0 \in \Omega$, and $0 < \varepsilon$ sufficiently small. 
Then the {\it generalized index of $M(\cdot)$ with respect to the counterclockwise 
oriented circle $C(z_0; \varepsilon)$}, $\wti{\ind}_{C(z_0; \varepsilon)} (M(\cdot))$, is defined by 
 \begin{align}
\begin{split}
\wti{\ind}_{C(z_0; \varepsilon)}(M(\cdot)) &= {\tr}_{\cH}\bigg(\f{1}{2\pi i} 
\ointctrclockwise_{C(z_0; \varepsilon)} d\zeta \, 
\ol{M'(\zeta)} M(\zeta)^{-1}\bigg)     \lb{4.15} \\
& = {\tr}_{\cH}\bigg(\f{1}{2\pi i} 
\ointctrclockwise_{C(z_0; \varepsilon)} d\zeta \, 
M(\zeta)^{-1} \ol{M'(\zeta)}\bigg). 
\end{split}
\end{align}  
(Of course, $\wti{\ind}_{C(z_0; \varepsilon_0)} (M(\cdot)) = 0$, unless, $z_0 \in \cD_0$.)
\end{definition}

The main objective of this paper is to show that this notion of generalized index applies to
Dirichlet-to-Neumann maps associated to non-self-adjoint Schr\"odinger operators in 
Section \ref{sec3} and to abstract operator-valued Weyl--Titchmarsh functions or Donoghue-type 
$M$-functions in Section \ref{sec4}.
It will also turn out that the generalized index is integer-valued in both of these applications.

\section{Schr\"{o}dinger Operators with Complex Potentials and Dirichlet-to-Neumann Maps} \lb{sec3}

In this section we discuss applications to Schr\"odinger operators with bounded, complex-valued potentials on bounded Lipschitz domains. In particular, we consider Krein-type resolvent formulas 
and compute the generalized index associated to underlying (energy parameter dependent) 
Dirichlet-to-Neumann maps. 

\begin{hypothesis}\lb{h4.1}
Let $\Omega\subset\bbR^n$, $n\geq 2$, be a bounded Lipschitz domain and let $q\in L^\infty(\Omega)$ be a complex-valued potential. 
\end{hypothesis}

Assuming Hypothesis \ref{h4.1}, we  consider the Schr\"{o}dinger differential expression
\begin{equation}
 \cL=-\Delta + q, 
\end{equation}
and its formal adjoint
\begin{equation}
 \widetilde\cL=-\Delta + \ol q.
\end{equation}
For our purposes, it is convenient to work with operator realizations of $\cL$ and $\widetilde\cL$ in $L^2(\Omega)$ which are defined via boundary conditions on functions from the space
\begin{equation}
 H^{3/2}_\Delta(\Omega):=\bigl\{f\in H^{3/2}(\Omega)\,\big|\,\Delta f\in L^2(\Omega)\bigr\},
\end{equation}
where for each $f\in H^{3/2}(\Omega)$, $\Delta f$ is understood in the sense of distributions. The space $H^{3/2}_\Delta(\Omega)$ equipped with the scalar product
\begin{equation}
 (f,g)_{H^{3/2}_\Delta(\Omega)}=(f,g)_{H^{3/2}(\Omega)}+(\Delta f,\Delta g)_{L^2(\Omega)},\quad f,g\in H^{3/2}_\Delta(\Omega),
\end{equation}
is a Hilbert space. According to \cite[Lemmas~3.1 and 3.2]{GM11}, the
Dirichlet trace operator defined on $C^\infty(\overline\Omega)$ admits a continuous surjective extension 
\begin{equation}\label{gdex}
\gamma_D:H^{3/2}_\Delta(\Omega)\rightarrow H^1(\partial\Omega), 
\end{equation}
and the Neumann trace  operator defined on $C^\infty(\overline\Omega)$ admits a continuous surjective extension 
\begin{equation}\label{gnex}
\gamma_N:H^{3/2}_\Delta(\Omega)\rightarrow L^2(\partial\Omega).
\end{equation}
For our investigations it is important to note that Green's Second Identity extends to functions in $H^{3/2}_\Delta(\Omega)$, that is, 
\begin{equation}\label{green32}
\begin{split}
 (\cL f,g)_{L^2(\Omega)} -(f,\widetilde \cL g)_{L^2(\Omega)} =(\gamma_Df,\gamma_N g)_{L^2(\partial\Omega)} -(\gamma_Nf,\gamma_D g)_{L^2(\partial\Omega)},&\\
f,g\in H^{3/2}_\Delta(\Omega).&
\end{split}
\end{equation}

Next, we introduce the Dirichlet operators associated to the differential expressions $\cL$ and $\widetilde\cL$.

\begin{hypothesis}\lb{h4.2}
In addition to the assumptions in Hypothesis \ref{h4.1}, let $A_D$ and $\widetilde A_D$ denote the Dirichlet operators associated to the differential expressions $\cL$ and $\widetilde\cL$ in 
$L^2(\Omega)$, that is, 
\begin{equation}\label{ad}
 A_D f=\cL f,\quad f\in\dom (A_D)=\bigl\{g\in H^{3/2}_\Delta(\Omega)\, \big|\,\gamma_D g=0\bigr\},
\end{equation}
and
\begin{equation}\label{adt}
 \widetilde A_D f=\widetilde\cL f,\quad f\in\dom \big(\widetilde A_D\big) 
 = \bigl\{g\in H^{3/2}_\Delta(\Omega)\, \big|\,\gamma_D g=0\bigr\}.
\end{equation}
\end{hypothesis}
In the special case $q\equiv 0$, the operator $A_D$ coincides with the self-adjoint free Dirichlet Laplacian on $\Omega$, which we denote by $A_D^{(0)}$:
\begin{equation}
A_D^{(0)}f = -\Delta f,\quad f\in \dom(A_D^{(0)}) = \big\{g \in H^{3/2}_\Delta(\Omega)\,\big|\,
\gamma_D g=0\big\} 
\end{equation}
(cf., e.g.,  \cite[Theorem~2.10 and Lemma~3.4]{GM08} or \cite[Theorem B.2]{JK95}). 
Clearly, $A_D$ (resp., $\widetilde A_D$) may be viewed as an additive perturbation of $A_D^{(0)}$ by the bounded potential $q$ (resp., $\ol q$). 
These facts lead to the following result.

\begin{proposition}\label{propad}
Assume Hypothesis \ref{h4.2}.  The Dirichlet operators $A_D$ and $\widetilde A_D$ are densely defined, closed operators in $L^2(\Omega)$ which are adjoint to each other, 
\begin{equation}\label{adj}
 A_D^*=\widetilde A_D.
\end{equation}
In addition, $A_D$ and $\widetilde A_D$ have compact resolvents. 
\end{proposition}

We note that \eqref{adj} also implies
\begin{equation}
z \in \rho(A_D)\,\text{ if and only if }\, {\ol z} \in \rho\big(\widetilde A_D\big).
\end{equation}

In light of the fact that the Dirichlet trace operator $\gamma_D$ 
maps $H^{3/2}_\Delta(\Omega)$ onto $H^1(\partial\Omega)$, it follows that for $z \in\rho(A_D)$ and $\varphi\in H^1(\partial\Omega)$ the boundary value problem
\begin{equation}\label{bvp}
 \cL f - z f=0,\quad\gamma_D f=\varphi,
\end{equation}
admits a unique solution $f_z \in H^{3/2}_\Delta(\Omega)$. Analogously, for 
$\wti z \in\rho\big(\widetilde A_D\big)$ and $\psi\in H^1(\partial\Omega)$, the boundary value problem
\begin{equation}\label{bvpt}
 \widetilde \cL g-\wti z g=0,\quad\gamma_D g=\psi,
\end{equation}
admits a unique solution $g_{\wti z}\in H^{3/2}_\Delta(\Omega)$. These observations imply that the solution operators and the Dirichlet-to-Neumann maps in the next definition
are well-defined.

\begin{definition}\label{pn}
Assume Hypothesis \ref{h4.2} and suppose $z \in\rho(A_D)$ and 
${\wti z} \in\rho\big(\widetilde A_D\big)$.  Let $f_z,g_{\wti z} \in H^{3/2}_\Delta(\Omega)$ 
denote the unique solutions of \eqref{bvp} and \eqref{bvpt} for $\varphi,\psi\in H^1(\partial\Omega)$, respectively. \\ 
$(i)$ The solution operators $P(z)$ and $\widetilde P({\wti z})$ associated to the boundary value problems \eqref{bvp} and \eqref{bvpt} are defined  by
\begin{equation}
 P(z)\varphi = f_z, \quad \widetilde P({\wti z})\psi = g_{\wti z},
\end{equation}
respectively. \\ 
$(ii)$ The (energy parameter dependent) Dirichlet-to-Neumann maps $D(z)$ and 
$\widetilde D({\wti z})$ associated to $\cL$ and $\widetilde\cL$ are defined  by 
\begin{equation}
 D(z)\varphi=\gamma_N f_z, \quad \widetilde D({\wti z})\psi=\gamma_N g_{\wti z},
\end{equation}
respectively.
\end{definition}

In the following, the solution operators $P(z)$ and $\widetilde P({\wti z})$ will often be regarded as densely defined operators from $L^2(\partial\Omega)$ into $L^2(\Omega)$,
and the Dirichlet-to-Neumann maps will be viewed as densely defined operators in $L^2(\partial\Omega)$. The next lemma collects relevant properties
of the solution operators and Dirichlet-to-Neumann maps, and its proof is based primarily on Green's Second Identity, \eqref{green32}. The arguments are almost the same  
as in the self-adjoint case, or in the abstract framework of boundary triples for dual pairs of operators (see \cite{MM02}),
and will not be repeated here. The reader is also referred to Steps 4--6 in the 
proof of Lemma~\ref{dt1} where similar methods are used. 

\begin{lemma}\label{lemlem}
Assume Hypothesis \ref{h4.2}.  For $z_1,z_2 \in\rho(A_D)$ and 
${\wti z_1},{\wti z_2} \in\rho\big(\widetilde A_D\big)$ the following identities hold: 
 \\[1mm]
$(i)$ The Poisson operator $P(z_1):L^2(\partial\Omega)\rightarrow L^2(\Omega)$ defined on the dense subspace $\dom (P(z_1))=H^1(\partial\Omega)$ is bounded and its adjoint 
  $P(z_1)^*\in\cB\big(L^2(\Omega),L^2(\partial\Omega)\big)$ is given by
  \begin{equation}
   P(z_1)^*=-\gamma_N \big(\widetilde A_D- \ol{z_1} I_{L^2(\Omega)}\big)^{-1}.
  \end{equation}
 $\widetilde{(i)}$ The Poisson operator $\widetilde P({\wti z_1}):L^2(\partial\Omega)\rightarrow L^2(\Omega)$ defined on the dense subspace 
  $\dom \big(\widetilde P({\wti z_1})\big)=H^1(\partial\Omega)$ is bounded and its adjoint 
  $\widetilde P({\wti z_1})^*\in\cB\big(L^2(\Omega),L^2(\partial\Omega)\big)$ is given by
  \begin{equation}
   \widetilde P({\wti z_1})^*=-\gamma_N(A_D- {\ol{\wti z_1}} I_{L^2(\Omega)})^{-1}.
  \end{equation}
  $(ii)$ For all $\varphi\in H^1(\partial\Omega)$ one has 
  \begin{equation}
   P(z_1)\varphi=\bigl(I_{L^2(\Omega)}+(z_1-z_2)(A_D- z_1 I_{L^2(\Omega)})^{-1}\bigr)P(z_2)\varphi.
  \end{equation}
 $\widetilde{(ii)}$ For all $\psi\in H^1(\partial\Omega)$ one has 
  \begin{equation}
   \widetilde P({\wti z_1})\psi=\bigl(I_{L^2(\Omega)}+({\wti z_1}-{\wti z_2}) 
   \big(\widetilde A_D- {\wti z_1} I_{L^2(\Omega)}\big)^{-1}\bigr)\widetilde P({\wti z_2})\psi.
  \end{equation}
 $(iii)$ The Dirichlet-to-Neumann map $D(z_1):L^2(\partial\Omega)\rightarrow L^2(\partial\Omega)$ defined on the dense subspace $\dom (D(z_1))=H^1(\partial\Omega)$ is a closed operator in 
 $L^2(\partial\Omega)$ and it satisfies the identity
  \begin{equation}\label{iiia}
   \bigl(D(z_1)-D(\ol{z_2})\bigr)\varphi=(\ol{z_2}-z_1)\widetilde P(z_2)^*P(z_1)\varphi, 
   \quad \varphi\in H^1(\partial\Omega). 
  \end{equation}
In particular, one has
 \begin{equation}
   D(z_1)\varphi=D(\ol{z_2})\varphi+(\ol{z_2}-z_1)\widetilde P(z_2)^*\bigl(I_{L^2(\Omega)} 
   +(z_1 - z_2)(A_D- z_1 I_{L^2(\Omega)})^{-1}\bigr)P(z_2)\varphi, 
\end{equation} 
and for all $\varphi\in H^1(\partial\Omega)$, the map $z_1 \mapsto D(z_1)\varphi$ is holomorphic 
on $\rho(A_D)$.\\
 $\widetilde{(iii)}$ The Dirichlet-to-Neumann map $\widetilde D({\wti z_1}):L^2(\partial\Omega)\rightarrow L^2(\partial\Omega)$ defined on the dense subspace 
  $\dom \big(\widetilde D({\wti z_1})\big)=H^1(\partial\Omega)$ is a closed operator in $L^2(\partial\Omega)$ and it satisfies the identity
  \begin{equation}\label{iiib}
   \bigl(\widetilde D({\wti z_1})-\widetilde D(\ol{\wti z_2})\bigr)\psi=(\ol{\wti z_2}-{\wti z_1}) 
   P({\wti z_2})^*\widetilde P({\wti z_1})\psi, \quad \psi\in H^1(\partial\Omega). 
  \end{equation}
In particular, one has
  \begin{equation}
   \widetilde D({\wti z_1})\varphi=\widetilde D(\ol{\wti z_2})\psi+(\ol{\wti z_2}-{\wti z_1})
   P({\wti z_2})^*\bigl(I_{L^2(\Omega)}+({\wti z_1} - {\wti z_2})\big(\widetilde A_D- {\wti z_1} I_{L^2(\Omega)}\big)^{-1}\bigr)\widetilde P({\wti z_2})\psi, 
  \end{equation}
  and for all $\psi\in H^1(\partial\Omega)$, the map 
  ${\wti z_1}\mapsto \widetilde D({\wti z_1})\psi$ is holomorphic on $\rho \big(\widetilde A_D\big)$. 
\end{lemma}

As a useful consequence of Lemma~\ref{lemlem}, one obtains the following result.

\begin{corollary}\label{cori}
For all $\varphi,\psi\in H^1(\partial\Omega)$ one has 
\begin{equation}\label{aqaq}
 \frac{d}{dz} D(z)\varphi =-\widetilde P(\ol{z})^*P(z)\varphi, \quad 
 \frac{d}{d{\wti z}} \widetilde D({\wti z})\psi =-P(\ol{\wti z})^*\widetilde P({\wti z})\psi,
\end{equation}
and the densely defined bounded operators $D^\prime(z)=\frac{d}{dz} D(z)$ and 
$\widetilde D^\prime({\wti z})=\frac{d}{d{\wti z}} \widetilde D({\wti z})$
in $L^2(\partial\Omega)$  
admit continuous extensions 
\begin{equation}\label{yxyx}
 \overline {D^\prime(z)}=-\widetilde P(\ol{z})^*\overline{P(z)}\in\cB(L^2(\partial\Omega))
\end{equation}
and
\begin{equation}\label{xyxy}
 \overline{\widetilde D^\prime({\wti z})}=-P(\ol{\wti z})^*\overline{\widetilde P({\wti z})} 
 \in\cB(L^2(\partial\Omega)).
\end{equation}
The $\cB(L^2(\partial\Omega))$-valued functions $z \mapsto\overline{D^\prime(z)}$ and 
${\wti z}\mapsto \overline{\widetilde D^\prime({\wti z})}$ are analytic
on $\rho(A_D)$ and $\rho\big(\widetilde A_D\big)$, respectively, and finitely
meromorphic on $\bbC$.
\end{corollary}
\begin{proof}
By \eqref{iiia} and \eqref{iiib}, the derivatives $\frac{d}{dz} D(z)\varphi$ and 
$\frac{d}{d{\wti z}} \widetilde D({\wti z})\psi$ exist for
all $\varphi,\psi\in H^1(\partial\Omega)$ and have the form as in \eqref{aqaq}. It is also clear from 
Lemma~\ref{lemlem} that the operators 
\begin{equation}
 \widetilde P(\ol{z})^*P(z), \quad P(\ol{\wti z})^*\widetilde P({\wti z})
\end{equation}
are defined on the dense subspace $H^1(\partial\Omega)$, and both are bounded. Hence, the continuous extensions onto $L^2(\Omega)$ are given by \eqref{yxyx} and \eqref{xyxy}, 
respectively. From \eqref{yxyx} and Lemma~\ref{lemlem} we conclude for some $z_0\in\rho(A_D)$ and all $z \in\rho(A_D)$ that
\begin{align}
 \overline{D^\prime(z)}&=-\bigl(\bigl(I_{L^2(\Omega)}+(\ol{z} - \ol{z_0}) 
 \big(\widetilde A_D-\ol{z} I_{L^2(\Omega)}\big)^{-1}\bigr)\widetilde P(\ol{z_0})\bigr)^*\no\\
 &\qquad\times\bigl(I_{L^2(\Omega)}+(z - z_0)(A_D- z I_{L^2(\Omega)})^{-1}\bigr)\overline{P(z_0)}\no\\
 &\quad =- \widetilde P(\ol{z_0})^*\bigl(I_{L^2(\Omega)}+(z - z_0)(A_D 
 - z I_{L^2(\Omega)})^{-1}\bigr)\no\\
 &\qquad \times \bigl(I_{L^2(\Omega)}+(z - z_0)(A_D- z I_{L^2(\Omega)})^{-1}\bigr)\overline{P(z_0)},
\end{align}
which shows that $z \mapsto\overline{D^\prime(z)}$ is analytic on $\rho(A_D)$ and finitely meromorphic on $\bbC$ (cf. Examples~\ref{ex1} and \ref{ex2}).
\end{proof}

\begin{hypothesis}\lb{h4.7}
In addition to the assumptions in Hypothesis \ref{h4.1}, suppose $\Theta\in\cB(L^2(\partial\Omega))$, and let $A_{\Theta}$ and $\widetilde A_{\Theta^*}$ denote the Robin realizations of $\cL$ and 
$\widetilde \cL$ in $L^2(\Omega)$, 
\begin{equation}\label{at}
 A_\Theta f=-\Delta f+qf,\quad f\in \dom (A_\Theta)=\bigl\{g\in H^{3/2}_\Delta(\Omega)\,\big|\,\Theta\gamma_D g=\gamma_N g\bigr\},
\end{equation}
and
\begin{equation}\label{atstern}
 \widetilde A_{\Theta^*} f=-\Delta f+ {\ol  q} f,\quad f\in \dom \big(\widetilde A_{\Theta^*}\big) 
 = \bigl\{g\in H^{3/2}_\Delta(\Omega)\,\big|\,\Theta^*\gamma_D g=\gamma_N g\bigr\}.
\end{equation}
\end{hypothesis}
In connection with $A_{\Theta}$ and $\widetilde A_{\Theta^*}$, one obtains the following variant of Proposition~\ref{propad}:

\begin{proposition}
Assume Hypothesis \ref{h4.7}.  Then $A_\Theta$ and $\widetilde A_{\Theta^*}$ are closed operators in $L^2(\Omega)$ which are adjoint to each other, 
\begin{equation}\label{adjttt}
 A_\Theta^*=\widetilde A_{\Theta^*}.
\end{equation}
In addition, $A_\Theta$ and $\widetilde A_{\Theta^*}$ have compact resolvents. 
\end{proposition}

In the next preparatory lemma, we study the operators $D(z)-\Theta$ and 
$\widetilde D({\wti z})-\Theta^*$ and their inverses in $L^2(\partial\Omega)$. As will turn out, these operators play an important role in the Krein-type resolvent formulas and index formulas at the end of this section.

\begin{lemma}\label{dt1}
Assume Hypothesis \ref{h4.7}.  Let $z \in\rho(A_D)\cap\rho(A_\Theta)$, 
${\wti z}\in\rho\big(\widetilde A_D\big)\cap\rho\big(\widetilde A_{\Theta^*}\big)$, and let $D(z)$ and 
$\widetilde D({\wti z})$
be the Dirichlet-to-Neumann maps associated to $\cL$ and $\widetilde\cL$, respectively. Then the following assertions hold: \\ 
$(i)$ $D(z)-\Theta$ is boundedly invertible and the inverse is a compact operator in
$L^2(\partial\Omega)$, 
\begin{equation}
(D(z)-\Theta)^{-1}\in\cB_\infty(L^2(\partial\Omega)).
\end{equation}
Furthermore, the map $z \mapsto (D(z)-\Theta)^{-1}$ is analytic on $\rho(A_\Theta)$ and finitely meromorphic on $\bbC$. \\[1mm]  
$\widetilde{(i)}$ $\widetilde D({\wti z})-\Theta^*$ is boundedly invertible and the inverse is a compact operator in
$L^2(\partial\Omega)$, 
\begin{equation}
\bigl(\widetilde D({\wti z})-\Theta^*\bigr)^{-1}\in\cB_\infty(L^2(\partial\Omega)).
\end{equation}
Furthermore, the map ${\wti z}\mapsto \big(\widetilde D({\wti z})-\Theta^*\big)^{-1}$ is analytic 
on $\rho\big(\widetilde A_\Theta\big)$ and finitely meromorphic on $\bbC$. 
\end{lemma}
\begin{proof}
The proof of Lemma~\ref{dt1}\,($i$) is divided into seven separate steps. The proof of item 
$\widetilde{(i)}$ follows precisely the same strategy and is hence omitted. \\[1mm] 
{\it Step 1.}
It will be shown first that the operator $D(z)-\Theta$ is injective for any $z \in\rho(A_D)\cap\rho(A_\Theta)$. Assume that for some $\varphi\in H^1(\partial\Omega)$,
\begin{equation}
 \bigl(D(z)-\Theta\bigr)\varphi=0
\end{equation}
and let $f_z \in H^{3/2}_\Delta(\Omega)$ be the unique solution of the boundary value problem 
\begin{equation}
\begin{cases}
\cL f- z f=0,\\
\gamma_D f=\varphi.
\end{cases}
\end{equation}
Then one infers
\begin{equation}
 \Theta\gamma_D f_z = \Theta\varphi=D(z)\varphi=D(z)\gamma_D f_z 
 =\gamma_N f_z,
\end{equation}
and hence $f_z \in\dom (A_\Theta)$ with $A_\Theta f_z = z f_z$. As 
$z\in\rho(A_\Theta)$, one concludes $f_z=0$, and hence $\varphi=\gamma_D f_z=0$. \\[1mm] 
{\it Step 2.}
In order to see that $D(z)-\Theta$ maps onto $L^2(\partial\Omega)$, one recalls that the inverse of the Dirichlet-to-Neumann map $N(z)=D(z)^{-1}$, 
the Neumann-to-Dirichlet map, is well-defined for all $z\in\rho(A_D)\cap\rho(A_N)$, where $A_N$ denotes the Neumann realization of $\cL=-\Delta+q$,
\begin{equation}
 A_Nf=-\Delta f+qf,\quad f\in \dom(A_N)=\big\{g\in H^{3/2}_\Delta(\Omega)\,\big|\,\gamma_N g=0 \big\}.
\end{equation}
Moreover, it follows in the same way as in \cite[Proposition 4.6]{BL07} or \cite[Lemma 4.6]{BGMM16} that
\begin{equation}\label{ncomp}
 N(z)\in\cB_\infty(L^2(\partial\Omega)).
\end{equation}
For $z \in\rho(A_\Theta)\cap\rho(A_D)\cap\rho(A_N)$, the operator $I_{L^2(\partial\Omega)}-\Theta N(z)$ is injective. In fact, suppose that
$\varphi=\Theta N(z)\varphi$ for some $\varphi\in L^2(\partial\Omega)$
and choose $f_z \in H^{3/2}_\Delta(\Omega)$ such that $\cL f_z = z f_z$ and 
$\gamma_N f_z = \varphi$. Then 
\begin{equation}
 \gamma_N f_z = \varphi=\Theta N(z)\varphi=\Theta N(z)\gamma_N f_z 
 = \Theta\gamma_D f_z,
\end{equation}
and hence $f_z \in\dom (A_\Theta)$. As $z \in\rho(A_\Theta)$, one concludes that $f_z=0$, and therefore, $\varphi=\gamma_N f_z =0$. 

The fact \eqref{ncomp} and the assumption $\Theta\in\cB(L^2(\partial\Omega))$ imply $\Theta N(z)\in\cB_\infty(L^2(\partial\Omega))$
and since $I_{L^2(\partial \Omega)}-\Theta N(z)$ is injective, one concludes 
\begin{align}
\bigl(D(z)-\Theta\bigr)^{-1}&=N(z)\bigl((D(z)-\Theta)N(z)\bigr)^{-1}\no\\
&=N(z)\bigl(I_{L^2(\partial \Omega)}-\Theta N(z)\bigr)^{-1}\in\cB(L^2(\partial\Omega))
\end{align}
for all $z \in\rho(A_\Theta)\cap\rho(A_D)\cap\rho(A_N)$. Therefore, $(D(z)-\Theta)^{-1}$ is closed as an operator in $L^2(\partial\Omega)$ and since
$\ran((D(z)-\Theta)^{-1})=H^1(\partial\Omega)$, the operator $(D(z)-\Theta)^{-1}$ is also closed as an operator from $L^2(\partial\Omega)$ to
$H^1(\partial\Omega)$. This implies 
\begin{equation}
(D(z)-\Theta)^{-1}\in\cB\bigl(L^2(\partial\Omega),H^1(\partial\Omega)\bigr),
\end{equation}
and as $H^1(\partial\Omega)$ is compactly embedded in $L^2(\partial\Omega)$, one concludes 
\begin{equation}\label{compi}
(D(z)-\Theta)^{-1}\in\cB_\infty(L^2(\partial\Omega)), \quad 
z \in\rho(A_\Theta)\cap\rho(A_D)\cap\rho(A_N).  
\end{equation}
\noindent
{\it Step 3.} Let $z \in\rho(A_\Theta)\cap\rho(A_D)\cap\rho(A_N)$ and 
${\wti z}\in\rho\big(\widetilde A_{\Theta^*}\big)\cap\rho\big(\widetilde A_D\big) 
\cap\rho\big(\widetilde A_N\big)$.
One observes first that for $\varphi\in L^2(\partial\Omega)$ and $\psi\in L^2(\partial\Omega)$ the boundary value problems
\begin{equation}\label{1b}
\begin{cases}
 \cL f- z f=0,\\
  \gamma_N f-\Theta\gamma_D f=\varphi,
 \end{cases}
\end{equation}
and
\begin{equation}\label{2b}
\begin{cases}
 \widetilde\cL g- {\wti z} g=0,\\
  \gamma_N g-\Theta^*\gamma_D g=\psi,
 \end{cases}
\end{equation}
admit unique solutions in $H^{3/2}_\Delta(\Omega)$. In fact, since the operators
$(D(z)-\Theta)^{-1}$ and $\big(\widetilde D({\wti z})-\Theta^*\big)^{-1}$ are defined on 
$L^2(\partial\Omega)$, and map into $H^1(\partial\Omega)$, the boundary value problems
\begin{equation}
\begin{cases}
 \cL f- z f=0,\\
  \gamma_D f=(D(z)-\Theta)^{-1}\varphi,
 \end{cases}
\end{equation}
and
\begin{equation}
\begin{cases}
 \widetilde\cL g- {\wti z} g=0,\\
  \gamma_D g = \big(\widetilde D({\wti z})-\Theta^*\big)^{-1}\psi,
 \end{cases}
\end{equation}
admit unique solutions $f_z \in H^{3/2}_\Delta(\Omega)$ and 
$g_{\wti z}\in H^{3/2}_\Delta(\Omega)$. Since
\begin{equation}
 \gamma_N f_z - \Theta\gamma_D f_z = (D(z)-\Theta)\gamma_D f_z = \varphi, 
\end{equation}
and
\begin{equation}
 \gamma_N g_{\wti z}-\Theta^*\gamma_D g_{\wti z} 
 = \big(\widetilde D({\wti z})-\Theta^*\big)\gamma_D g_{\wti z} =\psi,
\end{equation}
it is clear that $f_z$ and $g_{\wti z}$ solve \eqref{1b} and \eqref{2b}, respectively. We shall denote the solution operators
corresponding to the boundary value problems \eqref{1b} and \eqref{2b} by $P_\Theta(z)$ and 
$\widetilde P_{\Theta^*}({\wti z})$, respectively, that is,
\begin{equation}\label{p1}
 P_\Theta(z):L^2(\partial\Omega)\rightarrow L^2(\Omega),\quad\varphi \mapsto f_z,
\end{equation}
and
\begin{equation}\label{p2}
 \widetilde P_{\Theta^*}({\wti z}):L^2(\partial\Omega)\rightarrow L^2(\Omega),\quad\psi 
 \mapsto g_{\wti z},
\end{equation}
where $f_z \in H^{3/2}_\Delta(\Omega)$ and $g_{\wti z} \in H^{3/2}_\Delta(\Omega)$ denote the unique solutions of \eqref{1b} and \eqref{2b}, respectively. \\[1mm] 
{\it Step 4.} We claim that for $z \in\rho(A_\Theta)\cap\rho(A_D)\cap\rho(A_N)$ and 
${\wti z} \in\rho\big(\widetilde A_{\Theta^*}\big)\cap\rho\big(\widetilde A_D\big)\cap\rho\big(\widetilde A_N\big)$ 
the operators $P_\Theta(z)$ and $\widetilde P_{\Theta^*}({\wti z})$ in \eqref{p1} and \eqref{p2}, respectively, are bounded, that is, 
\begin{equation}\label{qqq}
 P_\Theta(z)\in\cB\bigl(L^2(\partial\Omega),L^2(\Omega)\bigr), \quad 
 \widetilde P_{\Theta^*}({\wti z})\in\cB\bigl(L^2(\partial\Omega),L^2(\Omega)\bigr).
\end{equation}
In fact, in order to verify the assertion for $P_\Theta(z)$ let $\varphi\in L^2(\partial\Omega)$ and 
$k\in L^2(\Omega)$. Since $z \in\rho(A_\Theta)$ implies $\ol{z} \in\rho\big(\widetilde A_{\Theta^*}\big)$,  there exists $h\in\dom\big(\widetilde A_{\Theta^*}\big)$ such that 
\begin{equation}
k = \big(\widetilde A_{\Theta^*}- \ol{z} I_{L^2(\Omega)}\big)h.
\end{equation}
Thus one computes with the help of Green's Second Identity \eqref{green32}, the boundary condition $\gamma_N h=\Theta^*\gamma_D h$, and the definition of $P_\Theta(z)$, that
\begin{align}
 (P_{\Theta}(z)\varphi,k)_{L^2(\Omega)}
 &=\bigl(f_z, \big(\widetilde A_{\Theta^*}- \ol{z} I_{L^2(\Omega)}\big)h\bigr)_{L^2(\Omega)}\no\\
 &= \big(f_z,\widetilde\cL h\big)_{L^2(\Omega)}-( f_z,\ol{z} h)_{L^2(\Omega)}\no\\
 &= \big(f_z,\widetilde\cL h\big)_{L^2(\Omega)}-(\cL f_z,h)_{L^2(\Omega)}\no\\
 &= (\gamma_N f_z,\gamma_D h)_{L^2(\partial\Omega)}-(\gamma_D f_z,\gamma_N h)_{L^2(\partial\Omega)}\no\\
 &=(\gamma_N f_z,\gamma_D h)_{L^2(\partial\Omega)}-(\gamma_D f_z,\Theta^*\gamma_D h)_{L^2(\partial\Omega)}\no\\
 &=\bigl([\gamma_N f_z-\Theta\gamma_D f_z],\gamma_D h\bigr)_{L^2(\partial\Omega)}\no\\
 &=\bigl(\varphi,\gamma_D \big(\widetilde A_{\Theta^*}-\ol{z} I_{L^2(\Omega)}\big)^{-1}k\bigr)_{L^2(\partial\Omega)}.
 \end{align}
The above computation implies that $P_\Theta(z)^*$ is defined on all of $L^2(\Omega)$  and given by 
\begin{equation}
 P_\Theta(z)^* 
 = \gamma_D \big(\widetilde A_{\Theta^*}-\ol{z} I_{L^2(\Omega)}\big)^{-1},
\end{equation}
and since $P_\Theta(z)^*$ is automatically closed it follows that
\begin{equation}
 P_\Theta(z)^*\in\cB\bigl(L^2(\Omega),L^2(\partial\Omega)\bigr).
\end{equation}
Hence $P_\Theta(z)^{**}\in\cB(L^2(\partial\Omega),L^2(\Omega))$ and since $\dom (P_\Theta(z))=L^2(\partial\Omega)$ it follows that
$P_\Theta(z)$ and $P_\Theta(z)^{**}$ coincide. Consequently, 
$P_\Theta(z)\in\cB\bigl(L^2(\partial\Omega),L^2(\Omega)\bigr)$. 
The proof of the second assertion in \eqref{qqq} is completely analogous. \\[1mm] 
{\it Step 5.} It will be shown that the solution operators in \eqref{p1} and \eqref{p2} satisfy the identities
\begin{equation}\label{pp1}
 P_\Theta(z)=\bigl(I_{L^2(\Omega)}+(z - z_0)(A_\Theta- z I_{L^2(\Omega)})^{-1}\bigr)P_\Theta(z_0)
\end{equation}
for all $z, z_0\in\rho(A_\Theta)\cap\rho(A_D)\cap\rho(A_N)$, and
\begin{equation}\label{pp2}
 \widetilde P_{\Theta^*}({\wti z})=\bigl(I_{L^2(\Omega)}+({\wti z} - {\wti z_0}) \big(\widetilde A_{\Theta^*}-{\wti z} I_{L^2(\Omega)}\big)^{-1}\bigr)\widetilde P_{\Theta^*}({\wti z_0})
\end{equation}
for all ${\wti z}, {\wti z_0} \in\rho\big(\widetilde A_{\Theta^*}\big)\cap\rho\big(\widetilde A_D\big)
\cap\rho\big(\widetilde A_N\big)$,
respectively. We verify \eqref{pp1} and omit details of the analogous proof of \eqref{pp2}.
 Let $\varphi\in L^2(\partial\Omega)$ and let $f_{z_0}\in H^{3/2}_\Delta(\Omega)$ be the unique solution of the boundary value problem
 \begin{equation}\label{bvpp}
 \begin{cases}
 \cL f-  z_0 f=0,\\
  \gamma_N f - \Theta\gamma_D f=\varphi,
 \end{cases}
\end{equation}
so that $P_\Theta(z_0)\varphi=f_{z_0}$. Since $z \in\rho(A_\Theta)$, one can make use of the direct sum decomposition
\begin{equation}
 H^{3/2}_\Delta(\Omega)=\dom (A_\Theta)\,\dot+\,\bigl\{f\in H^{3/2}_\Delta(\Omega)\,\big|\,
 \cL f- z f=0\bigr\}
\end{equation}
and write $f_{z_0}$ in the form
\begin{equation}\label{fl}
 f_{z_0}=f_\Theta+f_z,
\end{equation}
where $f_\Theta\in\dom(A_\Theta)$ and $f_z \in H^{3/2}_\Delta(\Omega)$ satisfies 
$\cL f_z - z f_z = 0$. Since 
$\gamma_N f_\Theta-\Theta \gamma_D f_\Theta=0$, it follows
from \eqref{fl} that 
\begin{equation}
\gamma_N f_z - \Theta\gamma_D f_z = \gamma_N f_{z_0}-\Theta\gamma_D f_{z_0}=\varphi,
\end{equation}
and hence $f_z$ in \eqref{fl} is the unique solution of the boundary value problem
\begin{equation}\label{bvppp}
 \begin{cases}
 \cL f- z f=0, \\ \gamma_N f-\Theta\gamma_D f=\varphi,
 \end{cases}
\end{equation}
so that $P_\Theta(z)\varphi=f_z$.
As $f_z - f_{z_0}=-f_\Theta\in\dom(A_\Theta)$, one can choose $g\in L^2(\Omega)$ such that
\begin{equation}
 f_z - f_{z_0} = (A_\Theta- z I_{L^2(\Omega)})^{-1}g,
\end{equation}
and then one computes
\begin{align}
 (z - z_0)f_{z_0}&= z \bigl(f_z - (A_\Theta- z I_{L^2(\Omega)})^{-1}g \bigr)- z_0 f_{z_0} 
 \no\\
 &=\cL (f_z - f_{z_0})- z (A_\Theta- z I_{L^2(\Omega)})^{-1}g\no\\
 &=\cL (A_\Theta- z I_{L^2(\Omega)})^{-1}g- z (A_\Theta- z I_{L^2(\Omega)})^{-1}g\no\\
 &=g,
 \end{align}
which yields
\begin{align}
  P_\Theta(z)\varphi&=f_z  \no\\
 &= f_{z_0} + (A_\Theta- z I_{L^2(\Omega)})^{-1}g\no\\
 &=f_{z_0}+(z - z_0)(A_\Theta- z I_{L^2(\Omega)})^{-1}f_{z_0}\no\\
 &=\bigl(I_{L^2(\Omega)}+(z - z_0)(A_\Theta- z I_{L^2(\Omega)})^{-1}\bigr) P_\Theta(z_0)\varphi.
\end{align}
This establishes \eqref{pp1}; the proof of \eqref{pp2} is analogous. \\[1mm]
{\it Step 6.} Let $z \in\rho(A_\Theta)\cap\rho(A_D)\cap\rho(A_N)$ and 
${\wti z}\in\rho\big(\widetilde A_{\Theta^*}\big)\cap\rho\big(\widetilde A_D\big)\cap
\rho\big(\widetilde A_N\big)$. 
In this step we verify the identity
\begin{equation}\label{plpl}
 \bigl(D(z)-\Theta\bigr)^{-1}=\bigl(D(\ol{\wti z})-\Theta\bigr)^{-1}+(z - \ol{\wti z})\bigl(\widetilde P_{\Theta^*}({\wti z})\bigr)^* P_\Theta(z).
\end{equation}
Let $\varphi,\psi\in L^2(\partial\Omega)$ and let $f_z = P_\Theta(z)\varphi$ and 
$g_{\wti z} = \widetilde P_{\Theta^*}({\wti z})\psi$. Then $f_z$ satisfies
\begin{equation}
\begin{cases}
\cL f_z - z f_z = 0,\\
 \gamma_N f_z - \Theta\gamma_D f_z = \varphi,
\end{cases}
\end{equation} 
$g_{{\wti z}}$ satisfies
\begin{equation}
\begin{cases}
\widetilde\cL g_{\wti z} - {\wti z} g_{\wti z}=0,\\
 \gamma_N g_{\wti z} - \Theta^*\gamma_D g_{\wti z}=\psi,
\end{cases}
\end{equation}
and
\begin{equation}
 \gamma_D f_z = (D(z)-\Theta)^{-1}\varphi,\quad
 \gamma_D g_{\wti z}=\bigl(\widetilde D({\wti z})-\Theta^*\bigr)^{-1}\psi.
\end{equation}
Hence, one infers
\begin{align}
 &\bigl((D(z)-\Theta)^{-1}\varphi,\psi\bigr)_{L^2(\partial\Omega)} 
 - \bigl(\varphi,\big(\widetilde D({\wti z})-\Theta^*\big)^{-1}\psi\bigr)_{L^2(\partial\Omega)} \no\\
 &\quad =\bigl(\gamma_D f_z, [\gamma_N g_{\wti z} -\Theta^*\gamma_D g_{\wti z}]\bigr)_{L^2(\partial\Omega)} - \bigl([\gamma_N f_z - \Theta\gamma_D f_z],\gamma_D g_{\wti z}\bigr)_{L^2(\partial\Omega)} \no\\
 &\quad =(\gamma_D f_z,\gamma_N g_{\wti z})_{L^2(\partial\Omega)} -
          (\gamma_N f_z,\gamma_D g_{\wti z})_{L^2(\partial\Omega)} \no\\
 &\quad = \big(\cL f_z,g_{\wti z}\big)_{L^2(\Omega)} 
 - \big(f_z,\widetilde \cL g_{\wti z}\big)_{L^2(\Omega)}\no\\
 &\quad =(z f_z,g_{\wti z})_{L^2(\Omega)}-(f_z, {\wti z} g_{\wti z})_{L^2(\Omega)}\no\\
 &\quad =(z - \ol{\wti z})\bigl(P_\Theta(z)\varphi,\widetilde P_{\Theta^*}({}\wti z)\psi\bigr)_{L^2(\Omega)}\no\\
 &\quad =(z - \ol{\wti z})\bigl(\bigl(\widetilde P_{\Theta^*}({\wti z})\bigr)^* P_\Theta(z)\varphi,\psi\bigr)_{L^2(\partial\Omega)}.\label{ccc}
 \end{align}
In particular, for $z=\ol{\wti z}$, 
\begin{equation}
 \bigl((D(\ol{\wti z})-\Theta)^{-1}\varphi,\psi\bigr)_{L^2(\partial\Omega)} 
 = \bigl(\varphi,\big(\widetilde D({\wti z})-\Theta^*\big)^{-1}\psi\bigr)_{L^2(\partial\Omega)},
\end{equation}
and hence
\begin{equation}
 (D(\ol{\wti z})-\Theta)^{-1}=\bigl(\big(\widetilde D({\wti z})-\Theta^*\big)^{-1}\bigr)^*. \lb{ccA}
\end{equation}
Together with \eqref{ccc}, \eqref{ccA} implies that
\begin{equation}
 (D(z)-\Theta)^{-1}-(D(\ol{\wti z})-\Theta)^{-1}=(z - \ol{\wti z})\bigl(\widetilde P_{\Theta^*}({\wti z})\bigr)^* P_\Theta(z),
\end{equation}
yielding \eqref{plpl}. \\[1mm]
{\it Step 7.} For $z \in\rho(A_\Theta)\cap\rho(A_D)\cap\rho(A_N)$ and 
${\wti z}\in\rho\big(\widetilde A_{\Theta^*}\big) \cap \rho\big(\widetilde A_D\big) \cap 
\rho\big(\widetilde A_N\big)$ one obtains via \eqref{pp1} and \eqref{plpl} the identity
\begin{align}
 \bigl(D(z)-\Theta\bigr)^{-1}&=\bigl(D(\ol{\wti z})-\Theta\bigr)^{-1}\lb{4.69jj}\\
 &\quad+(z - \ol{\wti z})\bigl(\widetilde P_{\Theta^*}({\wti z})\bigr)^* 
 \bigl(I_{L^2(\Omega)}+(z - \ol{\wti z})(A_\Theta- z I_{L^2(\Omega)})^{-1}\bigr)
 P_\Theta(\ol{\wti z}).    \no
\end{align}
Here, the fact that $\ol{\wti z} \in \rho(A_\Theta)\cap\rho(A_D)\cap\rho(A_N)$ has been used.   It follows from \eqref{4.69jj} that the map 
\begin{equation}\label{gtz}
z \mapsto (D(z)-\Theta)^{-1} 
\end{equation}
is holomorphic on the set $\rho(A_\Theta)\cap\rho(A_D)\cap\rho(A_N)$ and that it admits an analytic continuation to the set $\rho(A_\Theta)$.
One also infers from \eqref{compi} that the values of this analytic continuation are compact operators in $L^2(\partial\Omega)$.
Moreover, the fact that $z \mapsto (A_\Theta- z I_{L^2(\Omega)})^{-1}$ is finitely meromorphic on $\bbC$ implies that the map in \eqref{gtz} is finitely meromorphic on $\bbC$ 
(cf. Example~\ref{ex2}), completing the proof of Lemma~\ref{dt1}.
\end{proof}

The next theorems contain the index formulas that constitute the main results in this section. 
To set the stage, we also verify Krein-type resolvent formulas which relate the inverses
$(A_\Theta- z I_{L^2(\Omega)})^{-1}$ and 
$\big(\widetilde A_{\Theta^*}- {\wti z} I_{L^2(\Omega)}\big)^{-1}$ with the resolvents of the Dirichlet realizations $A_D$ and $\widetilde A_D$, respectively.
For the self-adjoint case, such formulas are well-known and can be found, for example, in \cite{AB09}, \cite{BGMM16}, \cite{BL07}, \cite{BM14}, \cite{BMNW08}, \cite{GM08}, \cite{GM11}, \cite{Ma10}, \cite{PR09}, \cite{Po12}.
For dual pairs of elliptic differential operators we refer to \cite{BGW09}, and for a more abstract operator theory framework, see \cite{MM02} and \cite{MM03}. The present version 
is partly inspired by \cite[Theorem 6.16]{BL12} and can be regarded as a non-self-adjoint variant for dual pairs of Schr\"{o}dinger operators with complex-valued potentials.

\begin{theorem}\label{k1}
Assume Hypotheses \ref{h4.2} and \ref{h4.7}.  For $z \in\rho(A_D)\cap\rho(A_\Theta)$ the Krein-type resolvent formula
  \begin{equation}\label{kreinformel}
   (A_\Theta- z I_{L^2(\Omega)})^{-1}=(A_D- z I_{L^2(\Omega)})^{-1} 
   + P(z) (D(z)-\Theta)^{-1}\widetilde P(\ol{z})^*
  \end{equation}
holds, and 
\begin{equation}\label{indi}
\wti{\ind}_{C(z_0; \varepsilon)}(D(\cdot)-\Theta)=m_a(z_0;A_\Theta)-m_a(z_0;A_D), 
\quad z_0\in \bbC.
\end{equation}
\end{theorem}

\begin{theorem}\label{k2}
Assume Hypotheses \ref{h4.2} and \ref{h4.7}.  For ${\wti z} \in\rho\big(\widetilde A_D\big) 
\cap\rho\big(\widetilde A_{\Theta^*}\big)$ the Krein-type resolvent formula
\begin{equation}
\big(\widetilde A_{\Theta^*} - {\wti z} I_{L^2(\Omega)}\big)^{-1} = \big(\widetilde A_D - {\wti z} I_{L^2(\Omega)}\big)^{-1}+\widetilde P({\wti z}) \big(\widetilde D({\wti z}) 
- \Theta^*\big)^{-1}P(\ol{\wti z})^*
  \end{equation}
  holds and
\begin{equation}\label{indi2}
\wti{\ind}_{C(z_0; \varepsilon)} \big(\widetilde D(\cdot)-\Theta^*\big) 
= m_a\big(z_0;\widetilde A_{\Theta^*}\big)-m_a \big(z_0;\widetilde A_D\big),\quad z_0\in\bbC.
\end{equation}
\end{theorem}
\begin{proof}[Proof of Theorem~\ref{k1}]
Fix $z \in\rho(A_D)\cap\rho(A_\Theta)$.~One recalls that according to Lemma~\ref{dt1},
\begin{equation}
(D(z)-\Theta)^{-1}\in\cB_\infty(L^2(\partial\Omega)).
\end{equation}
Moreover, since 
\begin{equation}
\dom(P(z))=\dom\bigl(D(z)-\Theta\bigr)=\ran\bigl((D(z)-\Theta)^{-1}\bigr),
\end{equation}
the perturbation term 
\begin{equation}
P(z)\bigl(D(z)-\Theta\bigr)^{-1}\widetilde P(\ol{z})^*
\end{equation}
on the right-hand side of \eqref{kreinformel} is well-defined. 
Next, let $f\in L^2(\Omega)$ and consider the function 
 \begin{equation}
  h=(A_D- z I_{L^2(\Omega)})^{-1}f+P(z) (D(z)-\Theta)^{-1} 
  \widetilde P(\ol{z})^*f.
 \end{equation}
 We claim that $h\in H^{3/2}_\Delta(\Omega)$ satisfies the boundary condition
\begin{equation}\label{bc}
 \Theta\gamma_D h=\gamma_N f.
\end{equation}
First of all, it is clear that $h\in H^{3/2}_\Delta(\Omega)$ since $\dom (A_D)\subset H^{3/2}_\Delta(\Omega)$ by \eqref{ad} 
and $\ran(P(z))\subset H^{3/2}_\Delta(\Omega)$ by Definition \ref{pn}. 
In order to check \eqref{bc} one observes that
\begin{equation}\label{so}
 \gamma_D h=\gamma_D P(z)(D(z)-\Theta)^{-1}\widetilde P(\ol{z})^*f
           = (D(z)-\Theta)^{-1}\widetilde P(\ol{z})^*f, 
 \end{equation}
and 
\begin{equation}\label{undso}
\begin{split}
 \gamma_N h&=\gamma_N(A_D- z I_{L^2(\Omega)})^{-1}f+\gamma_N P(z) 
 (D(z)-\Theta)^{-1}\widetilde P(\ol{z})^*f\\
           &=-\widetilde P(\ol{z})^*f+D(z) (D(z)-\Theta)^{-1} 
           \widetilde P(\ol{z})^*f\\
           &=\Theta (D(z)-\Theta)^{-1}\widetilde P(\ol{z})^*f,
\end{split}
 \end{equation}
 where we have used Lemma~\ref{lemlem}\,$\widetilde{(i)}$ and the definition of the Dirichlet-to-Neumann map. At this point it is clear from \eqref{so} and \eqref{undso}
 that \eqref{bc} holds. Thus, one concludes $h\in\dom (A_\Theta)$ and hence it follows from 
 \begin{align}
   &(A_\Theta- z I_{L^2(\Omega)})h\no\\
   &\quad =(A_\Theta- z I_{L^2(\Omega)})\bigl((A_D- z I_{L^2(\Omega)})^{-1}f+P(z) 
   (D(z)-\Theta)^{-1}\widetilde P(\ol{z})^*f\bigr)\no\\
   &\quad =(\cL- z I_{L^2(\Omega)})(A_D- z I_{L^2(\Omega)})^{-1}f+(\cL- z I_{L^2(\Omega)})P(z) 
   (D(z)-\Theta)^{-1}\widetilde P(\ol{z})^*f\no\\
   &\quad =f
  \end{align}
that \eqref{kreinformel} holds as well.

Next we will verify that the map 
\begin{equation}\label{asdfg}
z \mapsto M(z) = D(z) - \Theta,\quad z \in\rho(A_D),
\end{equation}
satisfies the assumptions in Hypothesis~\ref{h3.5} with $\Omega=\bbC$ and $\cD_0=\sigma_p(A_D)\cup\sigma_p(A_\Theta)$. First, one recalls that the values of 
$D(\cdot)$ in \eqref{asdfg} are closed operators in $L^2(\partial\Omega)$ according to Lemma~\ref{lemlem}\,($iii$) and since $\Theta\in\cB(L^2(\partial\Omega))$ 
the same is true for the values of $M(\cdot)$.
It is also clear from Lemma~\ref{lemlem}\,($iii$) that 
\begin{equation}
\dom (M(z))=\dom(D(z))=H^1(\partial\Omega)
\end{equation}
is independent of $z$, that is, Hypothesis~\ref{h3.5}\,($i$) holds. Furthermore, it follows 
from Lemma~\ref{dt1}\,($i$)
that $M(z)^{-1}\in\cB(L^2(\partial\Omega))$ for all $z \in\bbC\backslash \cD_0$, and that $M(\cdot)^{-1}$ is analytic on $\bbC\backslash \cD_0$ and finitely meromorphic on $\bbC$. Hence, items ($ii$) and ($iii$) in Hypothesis~\ref{h3.5} are satisfied as well. Finally, the validity of items ($iv$) and ($v$) in  
Hypothesis~\ref{h3.5} follow from Lemma~\ref{lemlem}\,($iii$) and Corollary~\ref{cori}.

It remains to prove the index formula \eqref{indi}. Making use of Corollary~\ref{cori} and \eqref{kreinformel}, one obtains for  $0 < \varepsilon $ sufficiently small, 
\begin{align}
&\wti{\ind}_{C(z_0; \varepsilon)}(M(\cdot)) 
=  {\tr}_{L^2(\partial\Omega)}\bigg(\f{1}{2\pi i} 
\ointctrclockwise_{C(z_0; \varepsilon)} d\zeta \,  M(\zeta)^{-1} \overline{M'(\zeta)}\bigg)  \no \\
&\quad ={\tr}_{L^2(\partial\Omega)}\bigg(\f{1}{2\pi i} 
\ointctrclockwise_{C(z_0; \varepsilon)} d\zeta \,  \bigl(D(\zeta)-\Theta\bigr)^{-1} \overline{D'(\zeta)}\bigg)
\no \\
&\quad =-\f{1}{2\pi i} 
\ointctrclockwise_{C(z_0; \varepsilon)} d\zeta \, {\tr}_{L^2(\partial\Omega)}\Big( \bigl(D(\zeta)-\Theta\bigr)^{-1} \widetilde P(\overline\zeta)^*\overline{P(\zeta)}\Big) \no \\
&\quad =-\f{1}{2\pi i} 
\ointctrclockwise_{C(z_0; \varepsilon)} d\zeta \,  {\tr}_{L^2(\Omega)}\Big(\overline{P(\zeta)}\bigl(D(\zeta)-\Theta\bigr)^{-1} \widetilde P(\overline\zeta)^*\Big) \no \\
&\quad ={\tr}_{L^2(\Omega)}\bigg(-\f{1}{2\pi i} 
\ointctrclockwise_{C(z_0; \varepsilon)} d\zeta \,  P(\zeta)\bigl(D(\zeta)-\Theta\bigr)^{-1} \widetilde P(\overline\zeta)^*\bigg) \no \\
&\quad ={\tr}_{L^2(\Omega)}\bigg(-\f{1}{2\pi i} 
\ointctrclockwise_{C(z_0; \varepsilon)} d\zeta \,  \bigl((A_\Theta-\zeta I_{L^2(\Omega)})^{-1}-(A_D-\zeta I_{L^2(\Omega)})^{-1}\bigr)\bigg)
\no \\
&\quad =\tr_{L^2(\Omega)}(P(z_0;A_\Theta))-\tr_{L^2(\Omega)}(P(z_0;A_D))\no\\
&\quad =m_a(z_0;A_\Theta)-m_a(z_0;A_D), 
\end{align}  
where $P(z_0;A_\Theta)$ and $P(z_0;A_D)$ denote the Riesz projections onto the algebraic eigenspaces of $A_\Theta$ and $A_D$ corresponding to $z_0$;
cf. Example~\ref{ex1}.
\end{proof}

\section{Closed Extensions of Symmetric Operators and Abstract Weyl--Titchmarsh $M$-Functions}\label{sec4}

Let $B_1$ and $B_2$ be densely defined closed operators in a separable complex Hilbert space $\mathfrak H$ such that $\rho(B_1)\cap\rho(B_2)\not=\emptyset$
and consider the intersection $S=B_1\cap B_2$ of $B_1$ and $B_2$, which is a closed operator of the form
\begin{equation}\label{sop}
  Sf=B_1f=B_2f, \quad 
  \dom(S)= \{f\in\dom (B_1)\cap\dom (B_2) \,|\, B_1 f= B_2 f\}. 
\end{equation}

\begin{hypothesis}\lb{h5.1}
Assume that $S$ in \eqref{sop} is densely defined and symmetric in $\mathfrak H$ with equal 
deficiency indices.
Let $A_0$ be a fixed self-adjoint extension of $S$ in $\mathfrak H$, and assume that for $j=1,2$ the operators $A_0$ and $B_j$, as well as $A_0$ and $B_j^*$, are disjoint extensions of $S$,
that is,
\begin{equation}\label{discon}
 S=A_0\cap B_1=A_0\cap B_2=A_0\cap B_1^*=A_0\cap B_2^*.
\end{equation}
\end{hypothesis}

It follows from Hypothesis~\ref{h5.1} that both operators $B_1$ and $B_2$ are closed restrictions of the adjoint $S^*$ of $S$, and hence 
$B_1$ and $B_2$ can be parametrized with the help of a boundary triple for $S^*$ and closed parameters $\Theta_1$ and $\Theta_2$ in $\mathcal G$.
In the same manner, $B_1^*$ and $B_2^*$ are closed restrictions of $S^*$ and by \eqref{ats} they correspond to the parameters $\Theta_1^*$ and $\Theta_2^*$ in $\mathcal G$.
The assumption that for $j=1,2$ the operators $A_0$ and $B_j$, and $A_0$ and $B_j^*$ are disjoint extensions of $S$ implies that $\Theta_j$ and $\Theta_j^*$, $j=1,2$, 
are closed operators, and hence their domains are dense in $\mathcal G$.
We refer the reader to Appendix A for a brief introduction to the theory of boundary triples.

The following lemma is an immediate consequence of Proposition~\ref{propcon}, 
\eqref{bij}--\eqref{ats}, and \eqref{disjointq}

\begin{lemma}\label{deflem}
Assume that $B_1$, $B_2$, $S$ and $A_0$ satisfy Hypothesis~\ref{h5.1}. Then there exists a boundary triple 
 $\{\mathcal G,\Gamma_0,\Gamma_1\}$ for $S^*$, and densely defined closed operators 
 $\Theta_1,\Theta_2,\Theta_1^*,\Theta_2^*\in\mathcal C(\mathcal G)$, such that
 $A_0=S^*\upharpoonright\ker(\Gamma_0)$ and
\begin{equation}\label{bt12}
\begin{split}
 B_1&=S^*\upharpoonright\ker(\Gamma_1-\Theta_1\Gamma_0), \quad 
 B_2=S^*\upharpoonright\ker(\Gamma_1-\Theta_2\Gamma_0),\\
 B_1^*&=S^*\upharpoonright\ker(\Gamma_1-\Theta_1^*\Gamma_0), \quad \;
 B_2^*=S^*\upharpoonright\ker(\Gamma_1-\Theta_2^*\Gamma_0).
 \end{split}
\end{equation}
\end{lemma}

The next theorem is the main result of this section. Here we make use of the boundary triple in Lemma~\ref{deflem} and express the difference of the algebraic multiplicities
of a discrete eigenvalue of $B_1$ and $B_2$ with the help of the corresponding Weyl--Titchmarsh function $M(\cdot)$ and the parameters $\Theta_1$ and $\Theta_2$. Theorem~\ref{indithm}
can be viewed as an abstract variant of Theorem~\ref{k1}.

\begin{theorem}\label{indithm}
Assume that $B_1$, $B_2$, $S$ and $A_0$ satisfy Hypothesis~\ref{h5.1},
choose the boundary triple  $\{\mathcal G,\Gamma_0,\Gamma_1\}$ in Lemma~\ref{deflem},  
and $\Theta_1$, $\Theta_2$ such that \eqref{bt12} holds. Let $M(\cdot)$ be the Weyl--Titchmarsh function corresponding to $\{\mathcal G,\Gamma_0,\Gamma_1\}$ and assume that 
\begin{equation}\label{assi}
z_0\in\sigma_d(B_j)\cup\rho(B_j),\quad j=1,2,  \, \text{ and } \, z_0\in\sigma_d(A_0)\cup\rho(A_0).
\end{equation}
Then there exists $\varepsilon_0>0$ such that both functions $\Theta_j - M(\cdot)$, $j=1,2$, satisfy Hypothesis~\ref{h3.5} with $\Omega=D(z_0;\varepsilon_0)$
and $\mathcal D_0=\{z_0\}$, and the index formula
\begin{equation}\label{indiwow}
\wti{\ind}_{C(z_0; \varepsilon)} (\Theta_1 - M(\cdot)) -
\wti{\ind}_{C(z_0; \varepsilon)} (\Theta_2 - M(\cdot))
= m_a(z_0;B_1) -m_a(z_0; B_2)
\end{equation}
holds.
\end{theorem}
\begin{proof}
We verify  that the functions $\Theta_j - M(\cdot)$, $j=1,2$, satisfy items $(i)$--$(ii)$ and 
$(iv)$--$(v)$ of Hypothesis~\ref{h3.5} 
with $\Omega=D(z_0;\varepsilon_0)$ and $\mathcal D_0=\{z_0\}$. The proof of item $(iii)$ is more involved and will be given separately after Corollary~\ref{indicor}.
First, one observes that by the assumptions in \eqref{assi} one can choose $\varepsilon_0>0$ such that the punctured disc $D(z_0;\varepsilon_0)\backslash \{z_0\}$
is contained in the set $\rho(B_1)\cap\rho(B_2)\cap\rho(A_0)$. As $\Theta_j$, $j=1,2$ are densely defined closed operators by Lemma~\ref{deflem} and the values of the Weyl--Titchmarsh function
$M(\cdot)$ are bounded operators in $\mathcal G$, the functions
\begin{equation}
 \Theta_j - M(\cdot):D(z_0;\varepsilon_0)\backslash \{z_0\}\rightarrow\mathcal C(\mathcal G),\quad z\mapsto\Theta_j - M(z),\quad j=1,2,
\end{equation}
are well-defined and of the form as in Hypothesis~\ref{h3.5} with $\Omega=D(z_0;\varepsilon_0)$
and $\mathcal D_0=\{z_0\}$. It is also clear that $\dom(\Theta_j-M(z))=\dom(\Theta_j)$ is independent of $z\in  D(z_0;\varepsilon_0)\backslash \{z_0\}$
and that $(\Theta_j - M(z))^{-1}\in\mathcal B(\mathcal G)$ for all 
$z\in  D(z_0;\varepsilon_0)\backslash \{z_0\}$ by Theorem~\ref{kreinthm}\,$(i)$. Hence 
items $(i)$ and $(ii)$ in
Hypothesis~\ref{h3.5} are satisfied. Since the Weyl--Titchmarsh function $M(\cdot)$ is analytic 
on $\rho(A_0)$ one infers 
\begin{equation}
 \frac{d}{dz} (\Theta_j - M(z)) \varphi = -\frac{d}{dz} M(z)\varphi,\quad \varphi\in\dom(\Theta_j), 
 \quad z\in D(z_0;\varepsilon_0)\backslash \{z_0\}, 
\end{equation}
and hence
\begin{equation}\label{hurraq}
 \overline{(\Theta_j-M(z))^\prime}=-M^\prime(z),
\end{equation}
that is, items $(iv)$ and $(v)$ in Hypothesis~\ref{h3.5} hold (see \eqref{s3} and Lemma~\ref{mfm} for the fact that $M'(\cdot)$ is analytic on 
$D(z_0;\varepsilon_0)\backslash \{z_0\}$ and finitely meromorphic on $D(z_0;\varepsilon_0)$).

Next, we turn to the proof of the index formula. According to Theorem~\ref{kreinthm} one 
infers $z\in\rho(B_j)\cap\rho(A_0)$ if and only if $(\Theta_j - M(z))^{-1}\in\mathcal B(\mathcal G)$ 
for $j=1,2$ and Krein's formula 
\begin{equation}\label{k45}
 (B_j-zI_{\cH})^{-1}-(A_0-zI_{\cH})^{-1} = \gamma(z)(\Theta_j - M(z))^{-1}\gamma(\ol z)^*
\end{equation}
is valid for all $z\in\rho(B_j)\cap\rho(A_0)$, $j=1,2$; here $\gamma(\cdot)$ denotes the $\gamma$-field corresponding to the boundary triple 
$\{\mathcal G,\Gamma_0,\Gamma_1\}$. 
Let $P(z_0;B_j)$, $j=1,2$, and $P(z_0;A_0)$ be the Riesz projections onto the algebraic eigenspaces of $B_j$ and $A_0$ corresponding to $z_0$; 
since $A_0$ is self-adjoint the range of $P(z_0;A_0)$ coincides with $\ker(A_0-z_0)$.
Then it follows from Definition~\ref{d3.6}, \eqref{hurraq}, \eqref{mabl}, and \eqref{k45} in a similar 
manner as in the proof of Theorem~\ref{k1} that for $0<\varepsilon$ sufficiently small, 
\begin{equation}
\begin{split}
 \wti{\ind}_{C(z_0; \varepsilon)} (\Theta_j - M(\cdot))
 &= {\tr}_{\mathcal G}\bigg(\f{1}{2\pi i} 
\ointctrclockwise_{C(z_0; \varepsilon)} d\zeta \, 
(\Theta_j - M(\zeta))^{-1} \overline{(\Theta_j - M(\zeta))^\prime} \bigg)  \\   
 &= {\tr}_{\mathcal G}\bigg(-\f{1}{2\pi i} 
\ointctrclockwise_{C(z_0; \varepsilon)} d\zeta \, 
(\Theta_j - M(\zeta))^{-1} M^\prime(\zeta) \bigg)  \\ 
  &= -\f{1}{2\pi i} 
\ointctrclockwise_{C(z_0; \varepsilon)} d\zeta \, {\tr}_{\mathcal G}\big(
(\Theta_j - M(\zeta))^{-1} \gamma(\ol \zeta)^*\gamma(\zeta)\big)   \\
 &= -\f{1}{2\pi i} 
\ointctrclockwise_{C(z_0; \varepsilon)} d\zeta \,   {\tr}_{\mathfrak H}\big(
\gamma(\zeta)(\Theta_j - M(\zeta))^{-1} \gamma(\ol \zeta)^* \big)  \\
  &= {\tr}_{\mathfrak H}\bigg(-\f{1}{2\pi i} 
\ointctrclockwise_{C(z_0; \varepsilon)} d\zeta \, \big((B_j - \zeta I_{\cH})^{-1}-(A_0-\zeta I_{\cH})^{-1}\big) \bigg)  \\
&= {\tr}_{\mathfrak H} (P(z_0;B_j)) - {\tr}_{\mathcal G} (P(z_0;A_0))\\
&= m_a(z_0;B_j)-m_a(z_0;A_0), \quad j=1,2, 
\end{split}
 \end{equation} 
and hence
\begin{align}
 &\wti{\ind}_{C(z_0; \varepsilon)} (\Theta_1 - M(\cdot))- \wti{\ind}_{C(z_0; \varepsilon)} (\Theta_2 - M(\cdot))  \no \\
 &\quad = m_a(z_0;B_1)-m_a(z_0;A_0) -  m_a(z_0;B_2)+m_a(z_0;A_0)   \no \\
 &\quad = m_a(z_0;B_1) - m_a(z_0;B_2).
\end{align}
\end{proof}

In the next corollary, we discuss the special case that the closed operator $B_1$ is self-adjoint in $\mathfrak H$. In this case we
set $A_0=B_1$ and instead of Hypothesis~\ref{h5.1} it suffices to assume that the closed symmetric operator $S=A_0\cap B_2$ in \eqref{sop} is densely defined and 
that $S=A_0\cap B_2^*$ holds. Following Lemma~\ref{deflem} and Proposition~\ref{propcon} 
one obtains a boundary triple $\{\mathcal G,\Gamma_0,\Gamma_1\}$ for $S^*$,  
and densely defined closed operators $\Theta_2,\Theta_2^*\in\mathcal C(\mathcal G)$, such that
 $A_0=S^*\upharpoonright\ker(\Gamma_0)$ and 
\begin{equation}\label{ztzt}
B_2=S^*\upharpoonright\ker(\Gamma_1-\Theta_2\Gamma_0), \quad 
B_2^*=S^*\upharpoonright\ker(\Gamma_1-\Theta_2^*\Gamma_0).
\end{equation}

\begin{corollary}\label{indicor}
Let $B_1=A_0$, $B_2$, and $S=A_0\cap B_2$ be as above, and 
choose a boundary triple  $\{\mathcal G,\Gamma_0,\Gamma_1\}$ 
and $\Theta_2$ such that $A_0=S^*\upharpoonright\ker(\Gamma_0)$ and \eqref{ztzt} holds. Let $M(\cdot)$ be the Weyl--Titchmarsh function corresponding to $\{\mathcal G,\Gamma_0,\Gamma_1\}$
and assume that
\begin{equation}
z_0\in\sigma_d(B_2)\cup\rho(B_2)\cup\sigma_d(A_0)\cup\rho(A_0).
\end{equation}
Then there exists $\varepsilon_0>0$ such that the function $\Theta_2-M(\cdot)$ satisfies 
Hypothesis~\ref{h3.5} with $\Omega=D(z_0;\varepsilon_0)$
and $\mathcal D_0=\{z_0\}$, and the index formula
\begin{equation}
\wti{\ind}_{C(z_0; \varepsilon)} \big(\Theta_2 - M(\cdot)\big) 
= m_a\big(z_0;B_2\big) -m_a\big(z_0; A_0\big)
\end{equation}
holds.
\end{corollary}

It remains to show that the functions $\Theta_j - M(\cdot)$, $j=1,2$, satisfy 
Hypothesis~\ref{h3.5}\,$(iii)$. 
In the following considerations we discuss the general situation of unbounded closed operators $\Theta_1$ and $\Theta_2$ in Lemma~\ref{deflem} such that
\begin{equation}\label{b1}
 B_j = S^*\upharpoonright\ker(\Gamma_1-\Theta_j \Gamma_0), \quad 
 B_j^*=S^*\upharpoonright\ker(\Gamma_1-\Theta_j^*\Gamma_0), \quad j=1,2.  
\end{equation}
For the special
case of bounded operators $\Theta_1,\Theta_2\in\mathcal B(\mathcal G)$ the considerations simplify slightly and we refer the reader 
to Remark~\ref{boundedtheta} for more details.  We start with the following preliminary lemma.

\begin{lemma}\label{lemma1}
Let $\{\mathcal G,\Gamma_0,\Gamma_1\}$ be the boundary triple in Lemma~\ref{deflem}, let 
$\Theta_j \in\mathcal C(\mathcal G)$, $j=1,2$, be densely defined closed operators such that
\eqref{b1} holds, and consider the map 
\begin{equation}\label{gamt}
 \Gamma_0^{\Theta_j} = \Gamma_1-\Theta_j \Gamma_0,  \quad 
 \dom\big(\Gamma_0^{\Theta_j}\big) = \{ f\in\dom (S^*) \,|\, \Gamma_0 f\in\dom (\Theta_j)\}.
 \end{equation}
 Then the following assertions hold for $j=1,2$:
\\[1mm] 
 $(i)$ $\dom (B_j)=\ker\big(\Gamma_0^{\Theta_j}\big)$;\\[1mm]
 $(ii)$ $\ran (\Gamma_0^{\Theta_j})$ is dense in $\mathcal G$;\\[1mm]
 $(iii)$ the direct sum decomposition
\begin{equation}\label{dddeco}
 \dom\big(\Gamma_0^{\Theta_j}\big)=\dom (B_j)\,\dot 
 + \,\bigl(\ker(S^*-zI_{\cH})\cap\dom\big(\Gamma_0^{\Theta_j}\big)\bigr)
\end{equation}
holds for all $z\in\rho(B_j)$; \\[1mm]
 $(iv)$ $\dom\big(\Gamma_0^{\Theta_j}\big)$ is dense in $\dom (S^*)$ with 
respect to the graph norm.
 \end{lemma}
\begin{proof}
$(i)$ is a direct consequence of the definition of $\Gamma_0^{\Theta_j}$ in \eqref{gamt} and \eqref{b1}.
\\[1mm]
$(ii)$ In order to verify  that $\ran\big(\Gamma_0^{\Theta_j}\big)$ is dense in $\mathcal G$ consider first 
 the row operator $[-\Theta_j\,\, I_{\mathcal G}]:\mathcal G\times\mathcal G\rightarrow\mathcal G$
defined on the $\dom(\Theta_j)\times\mathcal G$ and note that by
\begin{equation}
\ran\bigl([-\Theta_j\,\, I_{\mathcal G}]\bigr)^\bot=\ker \left(\left[ \begin{matrix} 
-\Theta_j^* \\[1mm] I_{\mathcal G} \end{matrix} \right]\right)=\{0\}
\end{equation}
the range of $[-\Theta_j\,\, I_{\mathcal G}]$ is dense in $\mathcal G$. Hence it follows from 
\begin{equation}
 \Gamma_0^{\Theta_j}=\Gamma_1-\Theta_j \Gamma_0=[-\Theta_j\,\, I_{\mathcal G}]\left[ \begin{matrix} \Gamma_0 \\[1mm] \Gamma_1 \end{matrix} \right]
 \, \text{ and } \, 
 \ran\left(\left[ \begin{matrix} \Gamma_0 \\[1mm] \Gamma_1 \end{matrix} \right]\right) 
 = \mathcal G\times\mathcal G
\end{equation}
that $\ran\big(\Gamma_0^{\Theta_j}\big)$ is dense in $\mathcal G$. \\[1mm] 
$(iii)$ The inclusion $(\supset)$ in \eqref{dddeco} is clear from $(i)$. In order to verify the 
inclusion $(\subset)$ in \eqref{dddeco}, let $z\in\rho(B_j)$ and 
 $h\in\dom\big(\Gamma_0^{\Theta_j}\big)\subset\dom (S^*)$, and choose $k\in\dom(B_j)$
 such that
 \begin{equation}
  (S^*-zI_{\cH})h=(B_j-zI_{\cH})k.
 \end{equation}
 Since $S^*$ is an extension of $B_j$ it follows that $h-k\in\ker(S^*-zI_{\cH})$ and as $h\in\dom\big(\Gamma_0^{\Theta_j}\big)$ and 
 $k\in\ker\big(\Gamma_0^{\Theta_j}\big)\subset\dom\big(\Gamma_0^{\Theta_j}\big)$, hence 
 also $h-k\in\dom\big(\Gamma_0^{\Theta_j}\big)$. Thus, 
 \begin{equation}
  h=k+(h-k),\, \text{ where }\, k\in\dom(B_j),\,\,\, h-k\in\ker(S^*-zI_{\cH}) 
  \cap \dom\big(\Gamma_0^{\Theta_j}\big),
 \end{equation}
and hence the inclusion $(\subset)$ in \eqref{dddeco} is shown. The fact that the sum in \eqref{dddeco} is direct follows from the assumption $z\in\rho(B_j)$. \\[1mm]
$(iv)$ Since $\Gamma_0:\ker (S^*-zI_{\cH}) \rightarrow\mathcal G$, $z\in\rho(A_0)$, is an isomorphism with respect to the graph norm in $\ker(S^*-zI_{\cH})$ (which is equivalent 
to the norm in $\mathfrak H$), and 
since $\dom(\Theta_j)$ is dense in $\mathcal G$ 
we conclude that $\ker(S^*-zI_{\cH})\cap\dom\big(\Gamma_0^{\Theta_j}\big)$ is dense in $\ker(S^*-zI_{\cH})$ with respect to the graph norm. 
It follows from $(i)$ and the direct sum decomposition \eqref{dddeco} that 
$\dom\big(\Gamma_0^{\Theta_j}\big)$ is dense in $\dom (S^*)$ with 
respect to the graph norm.
\end{proof}

One observes that by Lemma~\ref{lemma1} the map 
\begin{equation}
 \Gamma_0^{\Theta_j}\upharpoonright\bigl(\ker(S^*-zI_{\cH})\cap\dom\big(\Gamma_0^{\Theta_j}\big)\bigr)\rightarrow\mathcal G,\quad z\in\rho(B_j),
\end{equation}
is injective and maps onto the dense subspace $\ran\big(\Gamma_0^{\Theta_j}\big)$. Hence,  
for $z\in\rho(B_j)$ fixed, and every $\varphi\in\ran\big(\Gamma_0^{\Theta_j}\big)$, there exists a unique $f_z\in \ker(S^*-zI_{\cH})\cap\dom\big(\Gamma_0^{\Theta_j}\big)$
such that $\Gamma_0^{\Theta_j}f_z=\varphi$. In analogy to the $\gamma$-field corresponding to $\{\mathcal G,\Gamma_0,\Gamma_1\}$ we define for $z\in\rho(B_j)$ the map 
\begin{equation}\label{gammathetai}
 \gamma_{\Theta_j}(z)\varphi=f_z,\quad\dom(\gamma_{\Theta_j}(z))
 = \ran\big(\Gamma_0^{\Theta_j}\big),
\end{equation}
where $f_z\in \dom\big(\Gamma_0^{\Theta_j}\big)\cap\ker(S^*-zI_{\cH})$ satisfies $\Gamma_0^{\Theta_j} f_z=\varphi$. In the next lemma some important properties
of the operators  $\gamma_{\Theta_j}(z)$ are collected. The methods in the proof are abstract analogs of the computations in Step 4 and Step 5 in the proof
of Lemma~\ref{dt1}.

\begin{lemma}\label{lemma2}
For all $z\in\rho(B_j)$ 
the operator $\gamma_{\Theta_j}(z)$ is densely defined and bounded from $\mathcal G$ into $\mathfrak H$. 
Furthermore, the identity
\begin{equation}\label{ida}
 \gamma_{\Theta_j}(z)\varphi=\bigl(I_{\mathcal G}+(z-\zeta)(B_j-zI_{\cH})^{-1}\bigr)\gamma_{\Theta_j}(\zeta)\varphi,\quad z,\zeta\in\rho(B_1),
\end{equation}
holds for all $\varphi\in\dom(\gamma_{\Theta_j}(z))=\dom(\gamma_{\Theta_j}(\zeta))=\ran(\Gamma_0^{\Theta_j})$, and extends by continuity to
\begin{equation}\label{idb}
 \overline{\gamma_{\Theta_j}(z)}=\bigl(I_{\mathcal G}+(z-\zeta)(B_j-zI_{\cH})^{-1}\bigr)\overline{\gamma_{\Theta_j}(\zeta)},\quad z,\zeta\in\rho(B_1).
\end{equation}
\end{lemma}
\begin{proof}
First of all it is clear from the definition of $\gamma_{\Theta_j}(z)$, $z\in\rho(B_j)$, in 
\eqref{gammathetai} and Lemma~\ref{lemma1}\,$(ii)$ 
that the operator $\gamma_{\Theta_j}(z)$ is densely defined in $\mathcal G$ and maps into $\mathfrak H$. Next we verify the identity \eqref{ida}.
Thus, let $z, \zeta\in\rho(B_1)$ and consider $\varphi\in\dom(\gamma_{\Theta_j}(z)) =\dom(\gamma_{\Theta_j}(\zeta))$. Then 
\begin{equation}\label{einsa}
 f_z=\gamma_{\Theta_j}(z)\varphi\in\ker(S^*-zI_{\cH})\cap\dom(\Gamma_0^{\Theta_j}), 
\end{equation}
and
\begin{equation}\label{einsb}
 f_\zeta=\gamma_{\Theta_j}(\zeta)\varphi\in\ker(S^*-\zeta I_{\cH})\cap\dom(\Gamma_0^{\Theta_j}), 
\end{equation}
and it follows from \eqref{dddeco} that there exists $f_j\in\dom(B_j)$ such that
\begin{equation}
 f_\zeta=f_j+f_z.
\end{equation}
As $f_z-f_\zeta=-f_j \in\dom(B_j)$ there exists $h\in\mathfrak H$ such that $f_z-f_\zeta=(B_j-zI_{\cH})^{-1}h$. It follows that
\begin{align}
 (z-\zeta)f_\zeta&=z\bigl(f_z-(B_j-zI_{\cH})^{-1}h\bigr) - \zeta f_\zeta  \no \\
                   &=S^*(f_z-f_\zeta)-z (B_j-zI_{\cH})^{-1}h  \no \\
                   &=S^*(B_j-zI_{\cH})^{-1}h-z (B_j-zI_{\cH})^{-1}h   \no \\
                   &=h
 \end{align}
and this implies
\begin{equation}
 f_z=f_\zeta +(B_j-zI_{\cH})^{-1}h=\bigl(I_{\mathcal G}+(z- \zeta)(B_j-zI_{\cH})^{-1}\bigr)f_\zeta.
\end{equation}
Together with \eqref{einsa}--\eqref{einsb} we conclude \eqref{ida}. 

Note that \eqref{idb} follows from \eqref{ida} and the fact that $\gamma_{\Theta_j}(z)$ and
$\gamma_{\Theta_j}(\zeta)$ are both continuous. In order to show the continuity of 
$\gamma_{\Theta_j}(z)$, $z\in\rho(B_j)$, it suffices to check 
that $\gamma_{\Theta_j}(z)^*\in\mathcal B(\mathfrak H,\mathcal G)$ since this yields $\overline{\gamma_{\Theta_j}(z)}=\gamma_{\Theta_j}(z)^{**}
\in\mathcal B(\mathcal G,\mathfrak H)$.
Fix
$z\in\rho(B_j)$ and recall from Lemma~\ref{deflem} that $B_j^*=S^*\upharpoonright\ker(\Gamma_1-\Theta^*_j\Gamma_0)$ and $\ol z\in\rho(B_j^*)$. 
Let $\varphi\in\dom(\gamma_{\Theta_j}(z))$, $f_z=\gamma_{\Theta_j}(z)\varphi\in\ker(S^*-zI_{\cH})\cap\dom(\Gamma_0^{\Theta_j})$ and $h\in\mathfrak H$, 
and choose $g\in\dom(B_j^*)$ such that $h=(B_j^*- \ol zI_{\cH})g$. Then one computes 
 \begin{align}
   (\gamma_{\Theta_j}(z)\varphi,h)_{\mathfrak H}&=\bigl(f_z,(B_j^*- \ol zI_{\cH})g\bigr)_{\mathfrak H}  \no \\
   &=(f_z,B_j^*g)_{\mathfrak H}-(z f_z,g)_{\mathfrak H}  \no \\
   &=(f_z,S^*g)_{\mathfrak H}-(S^* f_z,g)_{\mathfrak H}   \no \\
   &=(\Gamma_0 f_z,\Gamma_1 g)_{\mathcal G}-(\Gamma_1 f_z,\Gamma_0 g)_{\mathcal G}  \no \\
   &=(\Gamma_0 f_z,\Theta_j^*\Gamma_0 g)_{\mathcal G}
   - (\Gamma_1 f_z,\Gamma_0 g)_{\mathcal G}   \no \\
   &=\bigl(\Gamma_1 f_z-\Theta_j\Gamma_0f_z,-\Gamma_0 g\bigr)_{\mathcal G}   \no \\
   &=\bigl(\Gamma_0^{\Theta_j} f_z,-\Gamma_0 g\bigr)_{\mathcal G}   \no \\
   &=\bigl(\varphi,-\Gamma_0(B_j^*- \ol zI_{\cH})^{-1}h)_{\mathcal G}, 
 \end{align}
 and concludes $\gamma_{\Theta_j}(z)^*h=-\Gamma_0(B_j^*- \ol zI_{\cH})^{-1}h$, $h\in\mathfrak H$. In particular, since the adjoint operator $\gamma_{\Theta_j}(z)^*$
 is closed and defined on the whole space $\mathfrak H$ it follows that  $\gamma_{\Theta_j}(z)^*\in\mathcal B(\mathfrak H,\mathcal G)$. This completes the proof of Lemma~\ref{lemma2}.
\end{proof}

With the preparations in Lemma~\ref{lemma1} and Lemma~\ref{lemma2} we will now verify 
condition $(iii)$ in Hypothesis~\ref{h3.5} for the functions $\Theta_j-M(\cdot)$, $j=1,2$.
The proof of Proposition~\ref{iiiprop} is an abstract variant of the considerations in Step 6 and 7 in the proof of Lemma~\ref{dt1}.

\begin{proposition}\label{iiiprop}
Let $\{\mathcal G,\Gamma_0,\Gamma_1\}$ be the boundary triple in Lemma~\ref{deflem} with $A_0=S^*\upharpoonright\ker(\Gamma_0)$ and corresponding Weyl--Titchmarsh function 
$M(\cdot)$, and
let $\Theta_j\in\mathcal C(\mathcal G)$, $j=1,2$, be densely defined closed operators such that \eqref{bt12} holds.
Assume that
\begin{equation}\label{assi2}
z_0\in\sigma_d(B_j)\cup\rho(B_j),\quad j=1,2,\, \text{ and } \, z_0\in\sigma_d(A_0)\cup\rho(A_0).
\end{equation}
Then there exists $\varepsilon_0>0$ such that both functions $\Theta_j-M(\cdot)$, $j=1,2$, satisfy Hypothesis~\ref{h3.5}\,$(iii)$ with $\Omega=D(z_0;\varepsilon_0)$
and $\mathcal D_0=\{z_0\}$, that is, the functions
\begin{equation}
 \bigl(\Theta_j-M(\cdot)\bigr)^{-1}:D(z_0;\varepsilon_0)\backslash \{z_0\}\rightarrow\mathcal B(\mathcal G),\quad z\mapsto \bigl(\Theta_j-M(z)\bigr)^{-1}
\end{equation}
are analytic on $D(z_0;\varepsilon_0)\backslash \{z_0\}$ and finitely meromorphic on $D(z_0;\varepsilon_0)$.
\end{proposition}
\begin{proof}
Choose $\varepsilon_0>0$ as in the proof of Theorem~\ref{indithm}, so that the punctured disc $D(z_0;\varepsilon_0)\backslash \{z_0\}$
is contained in the set $\rho(B_j)\cap\rho(A_0)$, $j=1,2$. Let $z\in D(z_0;\varepsilon_0)\backslash \{z_0\}$ and fix $\zeta \in\rho(B_j)$.

Consider the map $\Gamma_0^{\Theta_j}$ in \eqref{gamt}, let $\gamma_{\Theta_j}(z)$ be as in \eqref{gammathetai} and
let $\varphi\in\ran\big(\Gamma_0^{\Theta_j}\big)$. Then
\begin{equation}
 f_z=\gamma_{\Theta_j}(z)\varphi \in\ker(S^*-zI_{\cH})\cap\dom\big(\Gamma_0^{\Theta_j}\big)
\end{equation}
satisfies $
\Gamma_0^{\Theta_j}f_z=\varphi$ and
since $M(z)\Gamma_0 f_z=\Gamma_1 f_z$ (see Definition~\ref{gwdef}) one finds 
\begin{equation}\label{pl}
 -\bigl(\Theta_j-M(z)\bigr)\Gamma_0f_z= -\Theta_j\Gamma_0 f_z+ \Gamma_1 f_z=\Gamma_0^{\Theta_j}f_z=\varphi,
\end{equation}
which implies 
\begin{equation}\label{plp}
 \bigl(\Theta_j-M(z)\bigr)^{-1}\varphi=-\Gamma_0 f_z;
\end{equation}
recall that $(\Theta_j-M(z))^{-1}\in\mathcal B(\mathcal G)$ for $z\in D(z_0;\varepsilon_0)\backslash \{z_0\}$ by 
Theorem~\ref{kreinthm}\,$(i)$.

Similarly, as ${\ol \zeta}\in\rho(B_j^*)$ and $B_j^*=S^*\upharpoonright\ker(\Gamma_1-\Theta_1^*\Gamma_0)$ the same argument as in the proof of Lemma~\ref{lemma1}\,$(ii)$ shows that
the range of 
\begin{equation}\label{gamti}
 \Gamma_0^{\Theta_j^*} = \Gamma_1-\Theta_j^*\Gamma_0,\quad 
 \dom\big(\Gamma_0^{\Theta_j^*}\big) = \{ f\in\dom (S^*) \, | \,\Gamma_0 f\in\dom (\Theta_j^*)\}, 
 \end{equation}
is dense in $\mathcal G$. 
The direct sum decomposition
\begin{equation}
 \dom\big(\Gamma_0^{\Theta_j^*}\big) = \dom (B_j^*)\,\dot + \,\bigl(\ker(S^*- \overline\zeta I_{\cH})\cap\dom\big(\Gamma_0^{\Theta_j^*}\big)\bigr)
\end{equation}
and $\dom (B_j^*)=\ker\big(\Gamma_0^{\Theta_j^*}\big)$
imply that for all $\psi\in\ran\big(\Gamma_0^{\Theta_j^*}\big)$ there 
exists a unique 
\begin{equation}
g_{{\ol \zeta}}\in\ker\big(S^*- {\ol \zeta}I_{\cH}\big)\cap\dom\big(\Gamma_0^{\Theta_j^*}\big)
\, \text{ such that } \, \Gamma_0^{\Theta_j^*}g_{{\ol \zeta}}=\psi.
\end{equation}
As in Lemma~\ref{lemma2} one verifies that the map $\gamma_{\Theta_j^*}({\ol \zeta}):\mathcal G\rightarrow\mathfrak H$, $\psi\mapsto g_{{\ol \zeta}}$ is densely defined and bounded, and, in particular,
the adjoint operator is bounded, that is, 
$(\gamma_{\Theta_j^*}({\ol \zeta}))^*\in\mathcal B(\mathfrak H,\mathcal G)$.  
The same argument as in \eqref{pl} shows that 
\begin{equation}\label{ghb}
 \bigl(\Theta_j^*-M({\ol \zeta})\bigr)^{-1}\psi=-\Gamma_0 g_{{\ol \zeta}}
\end{equation}
and a straightforward calculation using \eqref{plp}, \eqref{ghb}, \eqref{gamt}, \eqref{gamti} yields
 \begin{align}
   &\bigl((\Theta_j-M(z))^{-1}\varphi,\psi\bigr)_{\mathcal G}-\bigl((\Theta_j-M(\zeta))^{-1}\varphi,\psi\bigr)_{\mathcal G}   \no \\
   &\quad =\bigl((\Theta_j-M(z))^{-1}\varphi,\psi\bigr)_{\mathcal G}-\bigl(\varphi,(\Theta_j^* 
   - M({\ol  \zeta}))^{-1}\psi\bigr)_{\mathcal G}   \no \\
   &\quad =(-\Gamma_0 f_z,\Gamma_0^{\Theta_j^*}g_{{\ol \zeta}})_{\mathcal G}-(\Gamma_0^{\Theta_j}f_z,-\Gamma_0g_{{\ol \zeta}})_{\mathcal G}   \no \\
   &\quad =\bigl(-\Gamma_0 f_z,(\Gamma_1-\Theta_j^*\Gamma_0)g_{{\ol \zeta}}\bigr)_{\mathcal G}-\bigl((\Gamma_1-\Theta_j\Gamma_0)f_z,-\Gamma_0g_{{\ol \zeta}}\bigr)_{\mathcal G}   \no \\
   &\quad =(\Gamma_1 f_z,\Gamma_0 g_{{\ol \zeta}})_{\mathcal G}-(\Gamma_0 f_z,\Gamma_1 
   g_{{\ol  \zeta}})_{\mathcal G}    \no \\
   &\quad =(S^*f_z,g_{{\ol \zeta}})_{\mathfrak H}-(f_z,S^* g_{{\ol \zeta}} )_{\mathfrak H}   \no \\
   &\quad =(z f_z,g_{{\ol \zeta}})_{\mathfrak H}-(f_z, {\ol \zeta} g_{{\ol \zeta}} )_{\mathfrak H}   \no \\
   &\quad =(z-\zeta)(\gamma_{\Theta_j}(z)\varphi,\gamma_{\Theta_j^*}({\ol \zeta})\psi )_{\mathfrak H}.
 \end{align}
Hence
\begin{equation}
 (\Theta_j-M(z))^{-1}\varphi-(\Theta_j-M(\zeta))^{-1}\varphi=(z-\zeta)
 \bigr(\gamma_{\Theta_j^*}({\ol \zeta})\bigl)^*\gamma_{\Theta_j}(z)\varphi
\end{equation}
holds for all $\varphi\in\ran\big(\Gamma_0^{\Theta_j^*}\big)$ and with the help of the identities \eqref{ida} and \eqref{idb} in Lemma~\ref{lemma2} one obtains 
 \begin{align}
 \begin{split}
  &(\Theta_j-M(z))^{-1}   \\
  &\quad =(\Theta_j-M(\zeta))^{-1}+(z-\zeta)\bigl(\gamma_{\Theta_j^*}({\ol \zeta})\bigr)^*\bigl(I_{\mathcal G} 
  +(z-\zeta)(B_j-zI_{\cH})^{-1}\bigr)
  \overline{\gamma_{\Theta_j}(\zeta)}
 \end{split}
 \end{align}
 for all $z\in D(z_0;\varepsilon_0)\backslash \{z_0\}$. Since $z_0\in\sigma_d(B_j)\cup\rho(B_j)$ 
 the $\mathcal B(\mathfrak H)$-valued map $z\mapsto (B_j-zI_{\cH})^{-1}$
 is analytic on $D(z_0;\varepsilon_0)\backslash \{z_0\}$ and finitely meromorphic on $D(z_0;\varepsilon_0)$ by Example~\ref{ex1}. 
 As the operators $\overline{\gamma_{\Theta_j}(\zeta)}$ and
 $(\gamma_{\Theta_j^*}({\ol \zeta}))^*$ are bounded
 it follows from Example~\ref{ex2} that the same is true for the map 
 \begin{equation}
 z\mapsto \bigl(\gamma_{\Theta_j^*}({\ol \zeta})\bigr)^*(B_j-zI_{\cH})^{-1}\overline{\gamma_{\Theta_j}(\zeta)}.
 \end{equation}
Hence it follows that also the map $z\mapsto (\Theta_j-M(z))^{-1}$ is analytic on $D(z_0;\varepsilon_0)\backslash \{z_0\}$ and finitely meromorphic on $D(z_0;\varepsilon_0)$.
 This completes the proof of Proposition~\ref{iiiprop}.
\end{proof}

\begin{remark}\label{boundedtheta}
Assume that the closed operators $\Theta_j$, $j=1,2$, in \eqref{bt12} are bounded; this 
happens if and only if $\dom (S^*)=\dom (B_j)\,  +\,\dom (A_0)$ holds (see \eqref{rwe}).
In this case some of the previous considerations in Lemma~\ref{lemma1} and Lemma~\ref{lemma2}
slightly simplify. In particular, the map $\Gamma_0^{\Theta_j}$ in \eqref{gamt} is defined on $\dom (S^*)$ and maps onto $\mathcal G$. As a consequence, the operators
$\gamma_{\Theta_j}(z)$, $z\in\rho(B_j)$, are defined on $\mathcal G$ and the identities \eqref{ida} and \eqref{idb} are the same.
\end{remark}

\begin{remark}
A typical situation in which the closed operators $\Theta_j$, $j=1,2$, in \eqref{bt12} are unbounded is the following: Suppose that the deficiency indices of $S$ are infinite and that the resolvent difference 
\begin{equation}\label{resodiff}
 (B_j-zI_{\cH})^{-1}-(A_0-zI_{\cH})^{-1},\quad z\in\rho(A_0)\cap\rho(B_j),
\end{equation}
 is a compact operator. Then $\mathcal G$ is an infinite dimensional Hilbert space and it follows from \cite[Theorem 2]{DM91} that the closed operator $\Theta_j$ in $\mathcal G$ 
 has a compact resolvent, and hence is unbounded.
\end{remark}

\appendix

\section{Boundary Triplets, Weyl--Titchmarsh Functions, and Abstract Donoghue-type $M$-Functions} 
\lb{sA} 

The aim of this appendix is to give a brief introduction to boundary triples and their Weyl--Titchmarsh functions, and to establish the connection to 
abstract Donoghue-type $M$-functions that were studied, for instance, in \cite{Do65}, 
\cite{GKMT01}, \cite{GMT98}, \cite{GNWZ15}, \cite{GT00}, \cite{KO77}, \cite{KO78}, and \cite{LT77}.
In addition, we refer the reader to \cite{AB09}, \cite{AP04}, \cite{BL07}--\cite{BNMW14}, \cite{BGP08}, \cite{DHMS06}--\cite{DM95}, \cite{GG91}, \cite{HMM13}, \cite{Ma10}--\cite{Ry07}, 
for more details, applications, and references on boundary triples and their Weyl--Titchmarsh 
functions.

Let $\mathfrak H$ be a separable complex Hilbert space, let $S$ be a densely defined closed symmetric operator in 
$\mathfrak H$ and let $S^*$ be the adjoint of $S$. 
The notion of boundary triple (or boundary value space) appeared first in \cite{Br76} and \cite{Ko75}.

\begin{definition}\label{btdeff}
 A triple $\{\mathcal G,\Gamma_0,\Gamma_1\}$ is called a {\em boundary triple} for $S^*$ if $\mathcal G$ is a Hilbert space and 
 $\Gamma_0,\Gamma_1:\dom (S^*)\rightarrow\mathcal G$ are linear operators such that 
 \begin{equation}\label{gid}
  (S^*f,g)_{\mathfrak H}-(f,S^*g)_{\mathfrak H}=(\Gamma_1 f,\Gamma_0 g)_{\mathcal G}-(\Gamma_0 f,\Gamma_1 g)_{\mathcal G}
 \end{equation}
holds for all $f,g\in\dom(S^*)$ and the map $\Gamma=\bigl[\begin{smallmatrix}\Gamma_0\\[.5mm] \Gamma_1\end{smallmatrix}\bigr]:\dom(S^*)\rightarrow\mathcal G\times\mathcal G$ is onto.
\end{definition}

We note that a boundary triple for $S^*$ exists if and only if the deficiency indices of $S$ coincide, or, equivalently, if $S$ admits self-adjoint extensions in $\mathfrak H$.
A boundary triple (if it exists) is not unique (except in the trivial case $S=S^*$). Assume in the following that $\{\mathcal G,\Gamma_0,\Gamma_1\}$ is a boundary triple 
for $S^*$. Then 
\begin{equation}
 S=S^*\upharpoonright\ker(\Gamma)=S^*\upharpoonright(\ker(\Gamma_0)\cap\ker(\Gamma_1))
\end{equation}
holds and the maps $\Gamma_0,\Gamma_1:\dom(S^*)\rightarrow\mathcal G$ are continuous with respect to the graph of norm of $S^*$. A key feature of a boundary triple is
that all closed extensions of $S$ can be parametrized in an efficient way. More precisely, there is a one-to-one correspondence between the closed
extensions $A_\Theta\subset S^*$ of $S$ and the closed linear subspaces (relations) $\Theta\subset\mathcal G\times\mathcal G$ given by
\begin{equation}\label{bij}
 \Theta\mapsto A_\Theta=
 S^*\upharpoonright\{f\in\dom(S^*) \,|\, \{\Gamma_0 f,\Gamma_1 f\}\in\Theta\}.
\end{equation}
In the case where $\Theta$ in \eqref{bij} is (the graph of) an operator, the extension $A_\Theta$ is given by
\begin{equation}\label{bij2}
 A_\Theta=S^*\upharpoonright \ker(\Gamma_1-\Theta\Gamma_0).
\end{equation}
A particularly convenient feature is that the adjoint of $A_\Theta$ in \eqref{bij}--\eqref{bij2} is given by the extension that corresponds to the parameter $\Theta^*$, that is,
the identity
\begin{equation}\label{ats}
 (A_\Theta)^*=A_{\Theta^*}
\end{equation}
holds;
here the adjoint of linear relation $\Theta$ is defined in the same manner as the adjoint of a densely defined operator. It follows, in particular, 
that $A_\Theta$ is self-adjoint in $\mathfrak H$ if and only if the parameter $\Theta$ is self-adjoint in $\mathcal G$. In the following the self-adjoint extension
\begin{equation}
 A_0=S^*\upharpoonright\ker(\Gamma_0)
\end{equation}
of $S$ will play the role of a fixed extension. One notes that $A_0$ corresponds to the subspace $\Theta_0=\{0\}\times\mathcal G$ in \eqref{bij}; in addition, one observes  that the index $0$ 
corresponds to the subspace $\Theta_0$ and not to the zero operator in $\mathcal G$.

For our purposes it is convenient to have criteria available which ensure that $\Theta$ in 
\eqref{bij}--\eqref{bij2} is a (bounded) operator. We recall from \cite{DM91}, \cite{DM95} that
$\Theta$ is a closed operator if and only $A_\Theta$ and $A_0$ are disjoint, that is, 
\begin{equation}\label{disjointq}
 S=A_\Theta\cap A_0,
\end{equation}
and that $\Theta\in\mathcal B(\mathcal G)$ if and only if $A_\Theta$ and $A_0$ are disjoint and 
\begin{equation}\label{rwe}
\dom (S^*)=\dom (A_\Theta) \, \dot +\,\dom (A_0)
\end{equation}
holds.

Next we recall the definition of the $\gamma$-field and Weyl--Titchmarsh function corresponding to a boundary triple $\{\mathcal G,\Gamma_0,\Gamma_1\}$. For this purpose consider the
self-adjoint operator 
$A_0=S^*\upharpoonright\ker(\Gamma_0)$ and note
that for any $z\in\rho(A_0)$ the direct sum decomposition 
\begin{equation}\label{dedede}
 \dom(S^*)=\dom (A_0)\, \dot +\,\ker(S^*-zI_{\cH})=\ker(\Gamma_0)\, \dot +\,\ker(S^*-zI_{\cH})
\end{equation}
holds. This implies, in particular, that the restriction of the boundary map $\Gamma_0$ onto $\ker(S^*-zI_{\cH})$ is injective for all $z\in\rho(A_0)$.
Moreover, the surjectivity of $\Gamma:\dom(S^*)\rightarrow\mathcal G\times\mathcal G$ and \eqref{dedede} yield that the restriction $\Gamma_0\upharpoonright\ker(S^*-zI_{\cH})$
maps onto $\mathcal G$ and hence the inverse $(\Gamma_0\upharpoonright\ker(S^*-zI_{\cH}))^{-1}$ is a bounded operator defined on $\mathcal G$. This observation shows that the $\gamma$-field and Weyl--Titchmarsh function
in the next definition are well-defined and their values are bounded operators for all $z\in\rho(A_0)$.

\begin{definition}\label{gwdef}
 Let $\{\mathcal G,\Gamma_0,\Gamma_1\}$ be a boundary triple for $S^*$ and let $A_0=S^*\upharpoonright\ker(\Gamma_0)$. The $\gamma$-field $\gamma(\cdot)$ corresponding to
 $\{\mathcal G,\Gamma_0,\Gamma_1\}$ is defined by
 \begin{equation}
  \gamma:\rho(A_0)\rightarrow\mathcal B(\mathcal G,\mathfrak H),\quad z\mapsto\gamma(z) = (\Gamma_0\upharpoonright\ker(S^*-zI_{\cH}))^{-1},
 \end{equation}
 and the Weyl--Titchmarsh function $M(\cdot)$
corresponding to
 $\{\mathcal G,\Gamma_0,\Gamma_1\}$ is defined by
 \begin{equation}
  M:\rho(A_0)\rightarrow\mathcal B(\mathcal G),\quad z\mapsto M(z)=\Gamma_1 (\Gamma_0\upharpoonright\ker(S^*-zI_{\cH}))^{-1}.
 \end{equation}
\end{definition}

In the following let $\gamma(\cdot)$ and $M(\cdot)$ be the $\gamma$-field and Weyl--Titchmarsh function corresponding to a boundary triple $\{\mathcal G,\Gamma_0,\Gamma_1\}$ for $S^*$.
We recall some important properties of the functions $\gamma(\cdot)$ and $M(\cdot)$ which can be found, for instance, in \cite{BL12}, \cite{BGP08}, \cite{DM91}, \cite{DM95}. First of all we note that $\gamma(\cdot)$ and $M(\cdot)$ are both
analytic operator functions on $\rho(A_0)$ with values in $\mathcal B(\mathcal G,\mathfrak H)$ and $\mathcal B(\mathcal G)$, respectively. The adjoint of
$\gamma(z)$ is a bounded operator from $\mathfrak H$ into $\mathcal G$ of the form
\begin{equation}
 \gamma(z)^*=\Gamma_1(A_0-\ol zI_{\cH})^{-1}\in\mathcal B(\mathfrak H,\mathcal G).
\end{equation}
Furthermore, the important identities 
\begin{equation}\label{s1}
 \gamma(z)=\bigl(I_{\mathcal G}+(z- \zeta)(A_0-zI_{\cH})^{-1}\bigr)\gamma(\zeta)
\end{equation}
and
\begin{equation}\label{s2}
 M(z)-M(\zeta)^*=(z- {\ol \zeta})\gamma(\zeta)^*\gamma(z)
\end{equation}
hold for all $z, \zeta \in\rho(A_0)$. A combination of \eqref{s1} and \eqref{s2} shows
\begin{equation}\label{s3}
\begin{split}
 M(z)&=M(\zeta)^*+(z- {\ol \zeta})\gamma(\zeta)^*\bigl(I_{\mathcal G} 
 + (z-\zeta)(A_0-zI_{\cH})^{-1}\bigr)\gamma(\zeta)\\
 &=M(\zeta)^*+(z-{\ol \zeta})\gamma(\zeta)^*\gamma(\zeta) 
 + (z-\zeta)(z- {\ol \zeta})\gamma(\zeta)^*(A_0-zI_{\cH})^{-1}\gamma(\zeta).
 \end{split}
 \end{equation}
One observes that \eqref{s2} implies 
\begin{align} 
& M(z)^*=M(\ol z), \quad z\in\rho(A_0),  \\
& \frac{d}{dz} M(z)=\gamma(\ol z)^*\gamma(z),\quad z\in\rho(A_0),   \label{mabl}
\end{align}
and that 
\begin{equation}
\Im (M(z))=\frac{1}{2i}\bigl(M(z)-M(z)^*\bigr)=(\Im(z))\gamma(z)^*\gamma(z)\in\mathcal B(\mathcal G)
\end{equation}
is a uniformly positive (resp., uniformly negative) operator for $z\in\mathbb C^+$ 
(resp., $z\in\mathbb C^-$). Therefore, the Weyl--Titchmarsh function $M(\cdot)$ is a 
$\mathcal B(\mathcal G)$-valued
Riesz--Herglotz or Nevanlinna function (see \cite{DM91}, \cite{GNWZ15}, \cite{KK74}, \cite{LT77}), 
which, in addition, is uniformly strict (cf. \cite{DHMS06}). In particular, there exists a self-adjoint
operator $\alpha\in\mathcal B(\mathcal G)$ and a non-decreasing self-adjoint operator map 
$t\mapsto\Sigma(t)\in\mathcal B(\mathcal G)$ on $\mathbb R$
such that $M(\cdot)$ admits the integral representation
\begin{equation}
 M(z)=\alpha+\int_{\bbR}d\Sigma(t)\left(\frac{1}{t-z}-\frac{t}{1+t^2}\right),\quad z\in\rho(A_0),
\end{equation}
where $\int_{\bbR}d\Sigma(t)(1+t^2)^{-1}\in\mathcal B(\mathcal G)$. 

The next lemma follows from \eqref{s3} and the fact that the resolvent of $A_0$ and its derivatives are finitely meromorphic at a discrete eigenvalue $z_0$ of $A_0$ (cf. Examples~\ref{ex1} and \ref{ex2}).

\begin{lemma}\label{mfm}
 Let $\{\mathcal G,\Gamma_0,\Gamma_1\}$ be a boundary triple for $S^*$ with $A_0=S^*\upharpoonright\ker(\Gamma_0)$, and let $M(\cdot)$ be the corresponding Weyl--Titchmarsh function.
 If $z_0\in\sigma_d(A_0)\cup\rho(A_0)$ then $M$ and its derivatives $M^{(l)}(\cdot)$, $l\in\mathbb N$, are finitely meromorphic at $z_0$.
\end{lemma}

In the next proposition we provide a particular boundary triple for $S^*$ such that the corresponding Weyl--Titchmarsh function coincides with the 
abstract Donoghue-type $M$-function that was studied, for instance, in  \cite{GKMT01}, \cite{GMT98}, \cite{GNWZ15}. The construction in Proposition~\ref{propcon} can be found, for instance, in \cite[Proposition 4.1]{D95}.
For the convenience of the reader we provide a short proof.

\begin{proposition}\label{propcon}
 Let $S$ be a densely defined closed symmetric operator in $\mathfrak H$ with equal deficiency indices, fix a self-adjoint extension $A$ of $S$ in $\mathfrak H$ and decompose
 the elements $f\in\dom(S^*)$ according to the direct sum decomposition
 \begin{equation}
  \dom (S^*)=\dom(A)\,\dot +\,\mathcal N_i,\quad \mathcal N_i=\ker(S^*-iI_{\cH}),
 \end{equation}
 in the form $f=f_A+f_i$, $f_A\in\dom(A)$, $f_i\in\mathcal N_i$.
 Let $P_{\mathcal N_i}:\mathfrak H\rightarrow\mathcal N_i$ be the orthogonal projection onto $\mathcal N_i$ and let $\iota_{\mathcal N_i}:\mathcal N_i\rightarrow \mathfrak H$ be 
 the canonical embedding of $\mathcal N_i$ into $\mathfrak H$.
 
 Then $\{\cN_i,\Gamma_0,\Gamma_1\}$, where the boundary maps 
 $\Gamma_0,\Gamma_1:\dom(S^*)\rightarrow\cN_i$ are defined by
 \begin{equation}\label{eee}
  \Gamma_0 f= f_i  \, \text{ and } \, \Gamma_1 f=P_{\mathcal N_i}(A+i)f_A+if_i,
 \end{equation}
 is a boundary triple for $S^*$ with $A_0=S^*\upharpoonright\ker(\Gamma_0)=A$ and the corresponding Weyl--Titchmarsh function $M(\cdot)$  is given by
\begin{equation}
 M(z) = z I_{\cN_i}+ (z^2+1) 
 P_{ \mathcal N_i} (A-zI_{\cH})^{-1} 
  \iota_{ \mathcal N_i},\quad z\in\rho(A).
\end{equation}
\end{proposition}
\begin{proof} 
Let $f,g\in\dom(S^*)$ be decomposed in the form $=f_A+f_i$ and $g=g_A+g_i$, where $f_A,g_A\in\dom( A)$ and $f_i,g_i\in\mathcal N_i$. 
Since $A$ is self-adjoint in $\mathfrak H$  we have $(Af_A,g_A)_{\mathfrak H}=(f_A,A g_A)_{\mathfrak H}$ and it follows that
\begin{equation}\label{sd}
\begin{split}
 &(S^*f,g)_{\mathfrak H}-(f,S^*g)_{\mathfrak H}\\
 &\qquad=\bigl(Af_A+ i f_i,g_A+g_i\bigr)_{\mathfrak H}
-\bigl(f_A+f_i,A g_A + i g_i\bigr)_{\mathfrak H}\\
&\qquad=\bigl(A f_A+ i f_i,g_i\bigr)_{\mathfrak H}+(i f_i,g_A)_{\mathfrak H}
-\bigl(f_i,A g_A+ i g_i\bigr)_{\mathfrak H}-(f_A,i g_i)_{\mathfrak H}\\
 &\qquad=\bigl((A+iI_{\cH})f_A+ i f_i,g_i\bigr)_{\mathfrak H}
-\bigl(f_i,(A+iI_{\cH}) g_A+i g_i\bigr)_{\mathfrak H}.
\end{split}
\end{equation}
Moreover, it follows from  the definition of the boundary maps in \eqref{eee}  that
\begin{equation}\label{ds}
\begin{split}
&(\Gamma_1 f,\Gamma_0 g)_{\mathcal N_i}-(\Gamma_0 f,\Gamma_1 g)_{\mathcal N_i}\\
&\qquad=\bigl( P_{\mathcal N_i} (A+iI_{\cH})f_A+ i f_i, g_i\bigr)_{\mathcal N_i} -
\bigl(f_i, P_{\mathcal N_i} (A+iI_{\cH}) g_A+ i g_i\bigr)_{\mathcal N_i} \\
&\qquad=\bigl((A+iI_{\cH}) f_A+ i f_i, g_i\bigr)_{\mathfrak H}-
\bigl(f_i, (A+iI_{\cH})g_A + i g_i\bigr)_{\mathfrak H}.
\end{split}
\end{equation}
Therefore, by combining \eqref{sd} and \eqref{ds} we conclude
\begin{equation}
 (S^*f,g)_{\mathfrak H}-(f,S^*g)_{\mathfrak H}=(\Gamma_1 f,\Gamma_0 g)_{\mathcal N_i}-(\Gamma_0 f,\Gamma_1 g)_{\mathcal N_i},
\end{equation}
and hence the abstract Green's identity \eqref{gid} in Definition~\ref{btdeff} is satisfied.

Next we verify that the map 
\begin{equation}\label{surj}
 \Gamma=\left[\begin{matrix}\Gamma_0\\[1mm] \Gamma_1\end{matrix}\right]:\dom(S^*)\rightarrow\mathcal G\times\mathcal G
\end{equation}
is surjective.
To see  this consider $\varphi, \psi\in\mathcal N_i$, choose
$f_A\in\dom(A)$ such that
\begin{equation}
(A+iI_{\cH}) f_A=\psi- i \varphi,
\end{equation} 
and let $f=f_A+\varphi\in\dom( S^*)$. It follows from \eqref{eee} that
\begin{equation}
 \Gamma_0 f=\varphi\, \text{ and } \,  
 \Gamma_1 f= 
 P_{\mathcal N_i}(A+iI_{\cH})f_A+i\varphi=\psi.
\end{equation}
Hence the map in \eqref{surj} is onto and it follows that 
$\{\cN_i,\Gamma_0,\Gamma_1\}$ is a boundary triple for $S^*$. It is clear from the construction that
$A_0=S^*\upharpoonright\ker(\Gamma_0)=A$ holds.

It remains to show that the Weyl--Titchmarsh function corresponding to the boundary triple  $\{\cN_i,\Gamma_0,\Gamma_1\}$ has the asserted form. For this
consider first $f_i\in\mathcal N_i$ and note that for $f=f_i$ the abstract boundary values in \eqref{eee}
are given by
\begin{equation}
\Gamma_0 f_i=f_i \, \text{ and } \, 
\Gamma_1 f_i=i f_i.
\end{equation}
Therefore, Definition \ref{gwdef} implies
\begin{equation}
\gamma(i):\mathcal N_i\rightarrow \mathfrak H,\quad f_i\mapsto \gamma(i)f_i=f_i,
\end{equation}
that is, $\gamma(i)$ is the canonical embedding of 
$\mathcal N_i$ into $\mathfrak H$,
\begin{equation}
\gamma(i)=\iota_{\mathcal N_i},
\end{equation}
and $\gamma(i)^*:\mathfrak H\rightarrow\mathcal N_i$ is the orthogonal projection onto $\mathcal N_i$, that is, $\gamma(i)^*=P_{\mathcal N_i}$. 
Furthermore, Definition \ref{gwdef} also implies
\begin{equation}
M(i):\mathcal G\rightarrow\mathcal G,\quad f_i\mapsto M(i)f_i=if_i,
\end{equation}
that is, $M(i)=i I_{\cN_i}$. Next, it follows from \eqref{s3} with $\zeta=i$ that
\begin{equation}
 M(z)= z I_{\cN_i} + (z^2+1)  P_{ \mathcal N_i} (A-zI_{\cH})^{-1} \iota_{ \mathcal N_i}
\end{equation}
holds for all $z\in\rho(A)$, completing the proof of Proposition~\ref{propcon}.
\end{proof}

Finally, we recall a useful version of Krein's resolvent formula for the resolvents of the closed 
extensions $A_\Theta$ in \eqref{bij}--\eqref{bij2}, which also provides a correspondence between 
the spectrum of $A_\Theta$ inside the set $\rho(A_0)$ and the spectrum of $\Theta-M(\cdot)$.

\begin{theorem}\label{kreinthm}
 Let $\{\mathcal G,\Gamma_0,\Gamma_1\}$ be a boundary triple for $S^*$ with $A_0=S^*\upharpoonright\ker(\Gamma_0)$, and let $\gamma(\cdot)$ and $M(\cdot)$ be the corresponding
 $\gamma$-field and Weyl--Titchmarsh function, respectively. Let $A_\Theta\subset S^*$ be a closed extension of $S$ which corresponds to a closed operator or subspace 
 $\Theta$ as in \eqref{bij}--\eqref{bij2}. Then the following assertions hold for all $z\in\rho(A_0)$:
 \\[1mm] 
 $(i)$ $z\in\rho(A_\Theta)$ if and only if $0\in\rho(\Theta-M(z))$. \\[1mm] 
 $(ii)$ $z\in\sigma_j(A_\Theta)$ if and only if $0\in\sigma_j(\Theta-M(z))$, $j\in \{p,c,r\}$. \\[1mm] 
 $(iii)$ for all $z\in\rho(A_\Theta)\cap\rho(A_0)$,
 \begin{equation} 
 (A_\Theta-zI_{\cH})^{-1}=(A_0-zI_{\cH})^{-1}+\gamma(z)\bigl(\Theta-M(z)\bigr)^{-1}\gamma(\ol z)^*.
 \end{equation} 
\end{theorem}

\medskip
\noindent {\bf Acknowledgments.} 
F.G.\ gratefully acknowledges kind invitations to the Institute for Numerical Mathematics 
at the Graz University of Technology, Austria, and to the Department of Mathematical Sciences 
of the Norwegian University of Science and Technology, Trondheim, for parts of May and June 2015. The extraordinary hospitality by Jussi Behrndt and Helge Holden at each institution, as well as the stimulating atmosphere at both places, are greatly appreciated. 


\end{document}